\documentclass[10pt, a4paper]{amsart}

\pdfoutput=1

\usepackage{mathrsfs, amsmath, amsthm, amssymb, stmaryrd}%
\usepackage{tikz}%
\usepackage[linktocpage]{hyperref}%
\usepackage[all]{xy}%
\xyoption{2cell}%
\UseAllTwocells%

\DeclareMathOperator{\id}{id}

\DeclareMathOperator{\op}{op}

\DeclareMathOperator{\Mod}{\mathbf{Mod}}
\DeclareMathOperator{\Rep}{Rep}

\DeclareMathOperator{\fp}{fp}
\DeclareMathOperator{\Hom}{Hom}
\DeclareMathOperator{\End}{End}
\DeclareMathOperator{\Spec}{Spec}

\DeclareMathOperator{\Tor}{Tor}
\DeclareMathOperator{\reg}{reg}
\DeclareMathOperator{\flt}{flat}

\DeclareMathOperator{\rk}{rk}
\DeclareMathOperator{\Sym}{Sym}

\DeclareMathOperator{\ev}{ev}

\DeclareMathOperator{\iso}{iso}

\DeclareMathOperator{\LFdi}{\LF^{\rk d}_{\iso}}

\DeclareMathOperator{\Rex}{\mathbf{Rex}}

\DeclareMathOperator{\SymMonCat}{SymMonCat}

\DeclareMathOperator{\PsCone}{PsCone}

\DeclareMathOperator{\Ind}{Ind}

\DeclareMathOperator{\Aff}{\mathbf{Aff}}

\DeclareMathOperator{\QCoh}{\mathbf{QCoh}}
\DeclareMathOperator{\VB}{\mathbf{VB}}

\DeclareMathOperator{\BGL}{\mathrm{BGL}}

\DeclareMathOperator{\fpqc}{\mathit{fpqc}}

\DeclareMathOperator{\colim}{colim}

\DeclareMathOperator{\Lex}{\mathbf{Lex}}

\DeclareMathOperator{\Cat}{\mathbf{Cat}}
\DeclareMathOperator{\Gpd}{\mathbf{Gpd}}

\DeclareMathOperator{\CAT}{\mathbf{CAT}}

\DeclareMathOperator{\Fun}{\mathrm{Fun}}
\DeclareMathOperator{\LF}{\mathrm{LF}}

\newcommand{\ca}[1]{\mathscr{#1}}

\newcommand{\Prs}[1]{\mathcal{P}\ca{#1}}

\DeclareMathOperator{\U}{O}

\newcommand{\ubar}[1]{\underline{#1\mkern-4mu}\mkern4mu }

\newcommand{\ten}[1]{\mathop{{\otimes}_{#1}}}

\newcommand{\pb}[1]{\mathop{{\times}_{#1}}}
\newcommand{\po}[1]{\mathop{{+}_{#1}}}

\newcommand{\defl}{\mathrel{\mathop:}=}

\theoremstyle{plain}
\newtheorem{thm}{Theorem}[subsection]
\newtheorem{prop}[thm]{Proposition}
\newtheorem{lemma}[thm]{Lemma}
\newtheorem{cor}[thm]{Corollary}

\theoremstyle{definition}
\newtheorem{example}[thm]{Example}
\newtheorem{rmk}[thm]{Remark}
\newtheorem{dfn}[thm]{Definition}

\newtheoremstyle{citing}{}{}{\itshape}{}{\bfseries}{.}{ }{\thmnote{#3}}
\theoremstyle{citing}

\newtheoremstyle{citingdfn}{}{}{}{}{\bfseries}{.}{ }{\thmnote{#3}}
\theoremstyle{citingdfn}

\numberwithin{equation}{section}

\keywords{Adams stacks, descent, weakly Tannakian categories}
\subjclass[2010]{14A20, 16T05, 18D10}

\author{Daniel Sch\"appi}
\thanks{The author gratefully acknowledges support through the Swiss National Foundation Fellowship P2SKP2\_148479}
\address{School of Mathematics and Statistics, University of
  Sheffield, Sheffield, S3 7RH, UK}
\email{d.schaeppi@sheffield.ac.uk}

\title{Descent via Tannaka duality}

\begin{document}

\begin{abstract}
 Given a diagram of schemes, we can ask if a geometric object over one of them can be built from descent data (usually objects of the same type over the various other schemes in the diagram, together with compatibility isomorphisms). Using the language of moduli stacks, we can rephrase this as follows: saying that descent problems for a given diagram have essentially unique solutions amounts to saying that the diagram in question is a (bicategorical) colimit diagram in a certain 2-category of stacks.

 In this paper we use generalized Tannaka duality to explicitly compute certain colimits in the 2-category of Adams stacks. Using this we extend recent results of Bhatt from algebraic spaces to Adams stacks and a result of Hall--Rydh to non-noetherian rings.

 We conclude the paper with a global version of the Beauville--Laszlo theorem, which states that a large class of schemes and stacks can be built by gluing the open complement of an effective Cartier divisor with the infinitesimal neighbourhood of the divisor.
\end{abstract}

\maketitle

\tableofcontents

\section{Introduction}\label{section:introduction}

\subsection{Overview}
 Given a commutative diagram
\begin{equation}\label{eqn:generic_descent_diagram}
\vcenter{\xymatrix{ X \ar[r] \ar[d] & Z \ar[d] \\ Y \ar[r] & W} }
\end{equation}
 of schemes, we can ask if various objects over $W$ (for example, vector bundles, princiapl $G$-bundles, elliptic curves, etc.) can be built from objects of the same kind over $Y$ and $Z$, together with an isomorphism of their pullbacks to $X$. In this case, we say that the diagram satisfies descent for the given kind of objects (and \emph{effective} descent if the resulting object over $W$ is unique up to unique isomorphism).

 For example, Milnor \cite[\S 2]{MILNOR} has shown that the diagram
\begin{equation}\label{eqn:milnor_descent}
\vcenter{ \xymatrix{\Spec(A) \ar[r] \ar[d] & \Spec(C) \ar[d] \\ \Spec(B) \ar[r] & \Spec(B \pb{A} C)}}
\end{equation}
 satisfies effective descent for vector bundles if one of the morphisms $C \rightarrow A$ or $B \rightarrow A$ is surjective. Another example comes from the Beauville--Laszlo theorem \cite{BEAUVILLE_LASZLO}, which states that the diagram
\begin{equation}\label{eqn:beauville_laszlo_descent}
\vcenter{ \xymatrix{ \Spec( A_{f} \ten{A} \widehat{A}) \ar[r] \ar[d] & \Spec(\widehat{A}) \ar[d] \\ \Spec(A_f) \ar[r] & \Spec(A)}}
\end{equation}
 satisfies effective descent for vector bundles if $f \in A$ is a non-zero divisor, $A_f$ is the localization of $A$ at $f$, and $\widehat{A} =\lim_n A \slash (f^n)$ is the completion of $A$ at the ideal $(f)$.

 Note that neither of these results follow directly from the fact that vector bundles satisfy faithfully flat descent: none of the morphisms in \eqref{eqn:milnor_descent} are assumed to be flat, and in \eqref{eqn:beauville_laszlo_descent}, the morphism $A \rightarrow \widehat{A}$ is not flat in general unless $A$ is noetherian. Even in the noetherian case, the fact that vector bundles satisfy faithfully flat descent only implies that a descent datum over the truncated simplicial diagram
\[
 \xymatrix{ \Spec\bigl((A_f \times \widehat{A})^{\otimes 3}\bigr) \ar@<1ex>[r] \ar[r] \ar@<-1ex>[r] & \Spec\bigl((A_f \times \widehat{A}) \otimes (A_f \times \widehat{A})\bigr) \ar@<0.75ex>[r] \ar@<-0.75ex>[r] & \ar[l] \Spec(A_f \times \widehat{A})}
\]
 gives rise to an essentially unique vector bundle over $\Spec(A)$. The fact that the datum of an isomorphism over the ring $\widehat{A} \ten{A} \widehat{A}$ is redundant requires an additional argument.

 One of our aims is to extend the above descent results from vector bundles to more general geometric objects, for example, principal $G$-bundles for certain flat affine group schemes $G$, elliptic curves, and formal groups. These are classified by the moduli stacks $BG$, $M_{\mathrm{ell}}$, and $M_{\mathrm{FG}}$ respectively (the latter is not an Artin stack since it is not finite dimensional: specifying a formal group law requires specifying infinitely many coefficients of a power series). These are all stacks for the $\fpqc$-topology, so the objects they classify satisfy faithfully flat descent, but as we just observed above, this does not imply the desired results directly.

 Using certain 2-categories of stacks, we can rephrase the descent problem for Diagram~\eqref{eqn:generic_descent_diagram}. To say that Diagram~\eqref{eqn:generic_descent_diagram} satisfies effective descent for $M$-objects for some stacks $M$ (that is, for morphisms to $M$) amounts to saying that the canonical functor
\[
 \Hom(W,M) \rightarrow \Hom(Y,M) \pb{\Hom(X,M)} \Hom(Z,M)
\]
 is an equivalence of groupoids. Thus if $M$ and all of $X$, $Y$, $Z$, and $W$ are objects of the same fixed 2-category of stacks, this is equivalent to saying that \eqref{eqn:generic_descent_diagram} is a bicategorical pushout diagram.

 We will use some of the results about colimits of Adams stacks established using generalized Tannaka duality in \cite{SCHAEPPI_COLIMITS} to explicitly compute certain colimit diagrams in the 2-category of Adams stacks (and thereby establish the corresponding descent results mentioned above for these diagrams).

 An \emph{Adams stack} is a stack on the $\fpqc$-site $\Aff_R$ which is quasi-compact, has an affine diagonal, and which has the \emph{resolution property}: the vector bundles form a generator of the category of quasi-coherent sheaves. The name comes from the fact that these are precisely the stacks associated to Adams Hopf algebroids (see \cite[Theorem~1.3.1]{SCHAEPPI_STACKS}). Note that these need not be Artin stacks since the source and target morphism of the corresponding affine groupoid need not be finitely presentable. Thomason \cite[\S 2]{THOMASON} has shown that many global quotient stacks $X \slash G$ are Adams stacks. The schemes which are also Adams stacks are precisely the quasi-compact semi-separated schemes with the resolution property. This includes for example all quasi-projective schemes \cite[\S 5.3]{EGAII} and all separated algebraic surfaces \cite{GROSS}.

 The main result of \cite{SCHAEPPI_COLIMITS} shows that the 2-category of Adams stacks has all bicategorical colimits, but the construction of these colimits is rather indirect. The result is established using the generalized Tannaka duality results from \cite{SCHAEPPI_STACKS, SCHAEPPI_INDABELIAN, SCHAEPPI_GEOMETRIC} which show that the 2-category of Adams stacks is contravariantly equivalent to the 2-category of weakly Tannakian categories. Our first aim is to carefully analyze this construction in order to facilitate the recognition of colimit diagrams of Adams stacks. The general result of this kind is Theorem~\ref{thm:colimit_recognition}. The following theorem is a useful special case.

 Given two functors $F \colon \ca{A} \rightarrow \ca{C}$ and $G \colon \ca{B} \rightarrow \ca{C}$ between small categories, the bicategorical pullback\footnote{This particular construction of a bicategorical pullback is also known as \emph{iso-comma-category}.} of $F$ along $G$ (in the category of small categories) is given by the category $\ca{A} \pb{\ca{C}} \ca{B}$ whose objects are triples $(A,f,B)$ where $A \in \ca{A}$, $B \in \ca{B}$, and $f \colon FA \rightarrow GB$ is an isomorphism. The morphisms $(A,f,B) \rightarrow (A^{\prime}, f^{\prime}, B^{\prime})$ in $\ca{A} \pb{\ca{C}} \ca{B}$ are pairs of morphisms $\alpha \colon A \rightarrow A^{\prime}$ in $\ca{A}$ and $\beta \colon B \rightarrow B^{\prime}$ in $\ca{B}$ making the evident square commutative.

 For an Adams stack $X$ we write $\VB^c(X)$ for the category of vector bundles of constant rank on $X$.

\begin{thm}\label{thm:pushout_recognition}
 Let
\[
 \xymatrix{X \ar[r] \ar[d] \xtwocell[1,1]{}\omit{\varphi} & Z \ar[d]^{g} \\ Y \ar[r]_-{f} & W}
\]
 be a diagram of Adams stacks over $R$ which commutes up to isomorpism $\varphi$. Then $(f,\varphi,g)$ exhibits $W$ as bicategorical pushout in the 2-category of Adams stacks if and only if the following two conditions hold:
\begin{enumerate}
 \item[(i)] The functor
\[
 \VB^c(W) \rightarrow \VB^c(Y) \pb{\VB^c(X)} \VB^c(Z)
\]
 induced by $(f^{\ast},\varphi^{\ast},g^{\ast})$ is an equivalence of categories;
\item[(ii)] If $M \in \QCoh_{\fp}(W)$ satisfies $f^\ast M \cong 0$ and $g^{\ast} M \cong 0$, then $M \cong 0$.
\end{enumerate}
\end{thm}

 \begin{rmk}
 On the level of objects, one direction of (i) follows from the fact that the Adams stack $\BGL_d$ is a  classifying stack for rank $d$ vector bundles. To give a rank $d$ vector bundle on the pushout $Y \po{X} Z$ amounts to giving a morphism of stacks $Y \po{X} Z \rightarrow \BGL_d$. By the universal property, this corresponds to giving vector bundles of rank $d$ on $Y$ and on $Z$, together with an isomorphism of their pullbacks to $X$. Note that we cannot use a similarly simple argument to identify general morphisms between vector bundles, since we are only dealing with stacks in groupoids (which can therefore only detect \emph{isomorphisms} between vector bundles, not arbitrary morphisms).
 \end{rmk}

 Using this result and some of its variations, we show that \eqref{eqn:milnor_descent} and \eqref{eqn:beauville_laszlo_descent} are bicategorical pushout diagrams in the 2-category of Adams stacks (see Theorem~\ref{thm:affine_pushout} respectively Theorem~\ref{thm:tubular}). Note that Ferrand \cite{FERRAND} showed that \eqref{eqn:milnor_descent} is a pushout diagram in the category of schemes and Bhatt \cite{BHATT} showed that \eqref{eqn:beauville_laszlo_descent} is a pushout diagram in the category of quasi-compact quasi-separated algebraic spaces. Neither of these statements follows directly from our result since we assume that all the objects in our category satisfy the resolution property.

 \subsection{Affine coproducts and affine completions}\label{section:introduction_applications}

 In \cite{BHATT}, Bhatt used a derived verion of Tannaka duality due to Lurie to compute various colimits in the category of quasi-compact quasi-separated algebraic spaces. Using the results of \cite{SCHAEPPI_COLIMITS} (which in turn rely on the non-dervied Tannaka duality results of \cite{SCHAEPPI_STACKS, SCHAEPPI_INDABELIAN, SCHAEPPI_GEOMETRIC}), we are able to show that some of these diagrams are also colimit diagrams in the 2-category of Adams stacks.

 Note that our results cannot be proven directly using dervied Tannaka duality. As pointed out by Hall and Rydh in the introduction of \cite{HALL_RYDH_TANNAKA}, the difficulty lies in the fact that \cite[Lemma~2.6]{BHATT} does not extend from algebraic spaces to stacks. In particular, not every tensor functor $D(B \mathbb{G}_m) \rightarrow D(k)$ of symmetric monoidal $(\infty,1)$-categories is induced by a morphism of stacks.

 Bhatt showed that for any quasi-compact quasi-separated algebraic space and any collection $(A_i)_{i \in I}$ of commutative rings, the canonical morphism
\[
X(\textstyle\prod\nolimits_{i \in I} A_i) \rightarrow \textstyle\prod\nolimits_{i \in I} X(A_i) 
\]
 is a bijection. As \cite[Example~8.3]{BHATT} shows, the same is \emph{not} true for the Adams stack $\mathrm{B} \mathbb{G}_m$. The reason for this is the existence of rings $A_i$ and invertible $A_i$-modules $L_i$ such that the minimal number of generators of the $L_i$ grows arbitrarily large. As the following result shows, this is in some sense the only problem that can arise in the context of Adams stacks.

 Let $X$ be an Adams stack over $R$, and let $(A_i)_{i \in I}$ be a family of commutative $R$-algebras. If there exist constants $\mu_d$, $d \in \mathbb{N}$, such that all locally free $A_i$-modules of rank $d$ can be generated by at most $\mu_d$ elements. In Theorem~\ref{thm:infinite_affine_coproducts} we will show that the canonical functor
\[
 X(\textstyle\prod\nolimits_{i \in I} A_i) \rightarrow \textstyle\prod\nolimits_{i \in I} X(A_i)
\]
 is an equivalence of categories under this assumption. Thus $\Spec(\prod_{i \in I} A_i)$ is the coproduct of the $\Spec(A_i)$ in the 2-category of Adams stacks. The condition is for example satisfied if all the $A_i$ are local rings, or if all the $A_i$ are noetherian rings of dimension at most $n$. Indeed, in the latter case, the minimal number of generators of a locally free $A_i$-module of rank $d$ is bounded by $d+n$  by the Forster--Swan Theorem.

 Our second example concerns completions. Let $A$ be an $R$-algebra which is $I$-adically complete for some ideal $I \subseteq A$ and let $X$ be an Adams stack over $R$. In Corollary~\ref{cor:adic_ring} we will show that the natural functor
\[
 X(A) \rightarrow \lim X(A \slash I^n)
\]
 is an equivalence of groupoids. In other words, this says that the canonical morphisms $\Spec(A \slash I^n) \rightarrow \Spec(A)$ exhibit $\Spec(A)$ as colimit of the chain
\[
 \xymatrix{ \ldots \ar[r] & \Spec(A \slash I^n) \ar[r] & \Spec(A \slash I^{n+1}) \ar[r] & \ldots }
\]
 in the 2-category of Adams stacks.

 This result was proved for all quasi-compact and quasi-separated algebraic spaces in \cite{BHATT}, and for all quasi-compact and quasi-separated Artin stacks and \emph{noetherian} rings $A$ in \cite{HALL_RYDH_TANNAKA}. As a corollary, we can extend \cite[Corollaries~1.2 and 4.4]{BHATT} on arc spaces to Adams stacks. Note that this does not follow from the result of \cite{HALL_RYDH_TANNAKA} since $A \llbracket x \rrbracket$ need not be noetherian for general $R$-algebras $A$.

\subsection{The global version of the Beauville--Laszlo theorem}

 We conclude our paper with the following global version of the Beauville--Laszlo theorem. Let $X$ be an Adams stack (for example, a quasi-compact semi-separated scheme with the resolution property) and let $Z \subseteq X$ be an effective Cartier divisor with open complement $U = X \setminus Z$ and sheaf of ideals $I \subseteq \U_X$. Let $\widehat{\U}_X \defl \lim_n \U_X \slash I^n$ be the completion of $\U_X$ at $I$ (computed in the category $\QCoh(X)$ of quasi-coherent sheaves on $X$, \emph{not} in the bigger category of $\U_X$-modules). We write $\widehat{Z}$ for the ``infinitesimal neighbourhood'' $\widehat{Z} \defl \Spec_X (\widehat{\U}_X)$ of $Z$ in $X$. In \S \ref{section:beauville_laszlo}, we will show that the diagram
\[
 \xymatrix{ U \pb{X} \widehat{Z} \ar[r] \ar[d] & \widehat{Z} \ar[d] \\ U \ar[r] & X}
\]
 is a (bicategorical) pushout diagram in the 2-category of Adams stacks.

 In order to prove this we need to check that the above diagram satisfies effective descent for vector bundles and their morphisms. Unfortunately, this cannot be deduced from the local version (where $X=\Spec(A)$ is an affine scheme and $Z$ is the closed subscheme defined by a non-zero divisor $f \in A$). The problem is that restriction to an open subscheme of $X$ does \emph{not} commute with the infinite limit used in the definiton of the completion $\widehat{\U}_X$. This problem is also discussed in \cite{BASSAT_TEMKIN}. As we will see in \S \ref{section:beauville_laszlo}, this is nevertheless true if we are working with Adams stacks (or, slightly more generally, with certain schemes that have enough flat objects, see Corollary~\ref{cor:global_beauville_laszlo}). 

 As pointed out in \cite[\S 5.2]{BASSAT_TEMKIN}, the ``infinitesimal neighbourhood'' as defined above can sometimes be quite large. Ben-Bassat and Temkin improve this situation by using formal schemes and Berkovich spaces, but this leaves the realm of Adams stacks.

\section*{Acknowledgments}
 I am very grateful for helpful discussions with Tom Bridgeland and I want to thank Oren Ben-Bassat for explaining some aspects of his work to me.

 This paper was written at the University of Sheffield during a stay made possible by the Swiss National Foundation Fellowship P2SKP2\_148479.
\section{Categorical background}\label{section:background} 

\subsection{Generalized Tannaka duality}\label{section:background_tannaka}
 Throughout we will need to use the language of 2-categories. We refer the reader to \cite{LACK_COMPANION} for the basic theory of 2-categories.

 In \cite{SCHAEPPI_COLIMITS}, the generalized Tannakian formalism of \cite{SCHAEPPI_STACKS, SCHAEPPI_INDABELIAN, SCHAEPPI_GEOMETRIC} was used to show that the 2-category $\ca{AS}$ of Adams stacks over a fixed commutative ring $R$ has all bicategorical colimits.

 A \emph{weakly Tannakian category} over $R$ is an $R$-linear finitely cocomplete\footnote{An $R$-linear category is finitely cocomplete if it has finite direct sums and cokernels.} symmetric monoidal category $\ca{A}$ whose category $\Ind(\ca{A})$ of ind-objects is abelian, which is generated by objects with duals, and which admits a \emph{fiber functor} (or \emph{covering}): there exists a faithful, flat, and right exact symmetric strong monoidal functor
\[
 w \colon \ca{A} \rightarrow \Mod_B
\]
 for some commutative $R$-algebra $B$ (see \cite[Definition~1.5]{SCHAEPPI_INDABELIAN}).

 Recall that a finitely cocomplete symmetric monoidal $R$-linear category $\ca{A}$ is called \emph{right exact symmetric monoidal} if the functor $A \otimes - \colon \ca{A} \rightarrow \ca{A}$ is right exact for all $A \in \ca{A}$. We denote the 2-category of right exact symmetric monoidal categories and right exact symmetric strong monoidal functors between them by $\ca{RM}$. The generalized Tannaka duality result we will use states that the pseudofunctor
\[
 \QCoh_{\fp}(-) \colon \ca{AS}^{\op} \rightarrow \ca{RM}
\]
 induces a biequivalence of 2-categories between $\ca{AS}^{\op}$ and the full sub-2-category of $\ca{RM}$ of weakly Tannakian categories. Indeed, the pseudofunctor is 2-fully faithful by \cite[Theorem~1.3.3]{SCHAEPPI_STACKS}, and by \cite[Theorem~1.6]{SCHAEPPI_INDABELIAN}, $\ca{A} \in \ca{RM}$ is weakly Tannakian if and only if it is equivalent to $\QCoh_{\fp}(X)$ for some Adams stack $X$.

 Computing colimits of Adams stacks therefore amounts to computing limits of weakly Tannakian categories. The main result of \cite{SCHAEPPI_COLIMITS} shows that the inclusion
\[
 \{ \text{weakly Tannakian categories} \} \rightarrow \ca{RM}
\]
 has a right biadjoint $T$ which associates a \emph{universal} weakly Tannakian category $T(\ca{A})$ to any $\ca{A} \in \ca{RM}$. Thus, to compute the limit of a diagram of weakly Tannakian categories, we need to proceed in two steps: first we need to compute its limit in $\ca{RM}$ and then we need to find the universal weakly Tannakian category associated to the resulting object of $\ca{RM}$. To compute such limits explicitly, we therefore need to get a good understanding of these two constructions: we review bicategorical limits of categories in \S \ref{section:background_limits} and we review universal weakly Tannakian categories in \S \ref{section:background_universal}.

\subsection{Bicategorical limits and colimits}\label{section:background_limits}

 Let $\ca{I}$ be a small category (usually \emph{not} $R$-linear), and let $\ca{K}$ be a (strict) 2-category (we could also work with bicategories, but all the examples we consider are strict 2-categories). As is customary, we refer to objects, morphisms, and morphisms between morphisms in $\ca{K}$ as 0-cells, 1-cells, respectively 2-cells. A good example to keep in mind is the 2-category of categories, where the 2-cells are given by natural transformations. From this we can form further examples by equipping categories with various algebraic structure (for example, we can consider categories which admit a certain type of colimits, or symmetric monoidal categories).

A \emph{diagram} of shape $\ca{I}$ in $\ca{K}$ is a pseudofunctor $D \colon \ca{I} \rightarrow \ca{K}$ (this means that composition and identities need not be preserved strictly, only up to coherent invertible 2-cells). A \emph{pseudocone} with vertex $K \in \ca{K}$ is a pseudonatural transformation from the constant pseudofunctor with value $K$ to $D$. Unraveling the definition, this amounts to giving a 1-cell $k_i \colon K \rightarrow D_i$ for each $i \in \ca{I}$ and an invertible 2-cell
\[
 \xymatrix{& D_i \ar[rd]^{D_{\varphi}} \\ K \rrtwocell\omit{<-3>\quad\kappa_{\varphi}}   \ar[ru]^{k_i} \ar[rr]_-{k_j} && D_j}
\]
 for each morphism $\varphi \colon i \rightarrow j$ in $\ca{I}$, subject to compatibility conditions with the 2-cells of the pseudofunctor $D$. A morphism of pseudocones $(k_i,\kappa_{\varphi}) \rightarrow (k^{\prime}_i,\kappa^{\prime}_{\varphi})$ is a \emph{modification}, that is, a collection of 2-cells $\theta_i \colon k_i \rightarrow k^{\prime}_i$ compatible with $\kappa_{\varphi}$ and $\kappa^{\prime}_{\varphi}$ for all morphisms $\varphi$ in $\ca{I}$. We denote the category of pseudocones of the diagram $D$ with vertex $K$ by $\PsCone(K,D)$.

 We say that a pseudocone $(\ell_i,\lambda_\varphi)$ with vertex $L$ exhibits $L$ as \emph{bicategorical limit} of $D$ if the functor
\[
 \ca{K}(K,L) \rightarrow \PsCone(K,D) 
\]
 which sends a 1-cell $f \colon K \rightarrow L$ to the pseudocone $(f\ell_i,f\lambda_{\varphi})$ is an equivalence of categories for all objects $K \in \ca{K}$. A \emph{bicategorical colimit} of $D$ is a bicategorical limit of $D$ in the opposite 2-category $\ca{K}^{\op}$ of $\ca{K}$ (that is, the 2-category obtained by reversing the directions of the 1-cells of $\ca{K}$). A Yoneda argument shows that bicategorical limits and colimits are unique up to essentially unique equivalence.

 The data involved in the definition of a diagram and a pseudocone simplify considerably if $\ca{I}$ is the free category on a graph, for example, the indexing category for a pullback or a countable chain. To give a diagram on such a category $\ca{I}$ simply amounts to picking an object $D_i$ for all $i \in \ca{I}$ and a 1-cell $D_{\varphi} \colon D_i \rightarrow D_j$ for each indecomposeable morphism $\varphi \colon i \rightarrow j$ in $\ca{I}$ (without any further data and axioms required). A pseudocone on such a diagram with vertex $K$ consists of a collection of 1-cells $k_i \colon K \rightarrow D_i$ and invertible 2-cells $\kappa_{\varphi}$ as above for each indecomposeable morphisms $\varphi$ (not subject to any conditions). This simplifies the computation of bicategorical limits of this shape, as the following examples shows.

 Let $\ca{K}=\Cat$ be the 2-category of small categories, functors, and natural transformations between them. From the definition of bicategorical limits and the 2-natural isomorphism $\Cat(\ast, \ca{A}) \cong \ca{A}$ it follows that the bicategorical limit of a diagram $D \colon \ca{I} \rightarrow \Cat$ must be equivalent to $\PsCone(\ast,D)$ (if it exists), and one can check that $\PsCone(\ast,D)$ is indeed a bicategorical limit.

 Using the above remark about indexing categories that are free on a graph, we get the following description of the bicategorical limit $\lim \ca{A}_i$ of a sequence
\[
 \xymatrix{ \ldots \ar[r]^-{F_n} & \ca{A}_n \ar[r]^-{F_{n-1}} & \ldots \ar[r]^-{F_1} & \ca{A}_1 \ar[r]^-{F_0} & \ca{A}_0 }
\]
 of functors. The objects of $\lim \ca{A}_i$ consist of collections $(A_i,\kappa_i)$ where $A_i \in \ca{A}_i$ and $\kappa_i \colon F_{i} A_{i+1} \rightarrow A_{i}$ are isomorphisms. The morphisms $(A_i, \kappa_i) \rightarrow (A^{\prime}_i,\kappa^{\prime}_i)$ are collections of morphisms $\theta_i \colon A_i \rightarrow A^{\prime}_i$ which are compatible with the $\kappa_i$.

 Similarly we find that the bicategorical pullback of a diagram
\[
 \xymatrix{\ca{A} \ar[r]^-{F} & \ca{C} & \ar[l]_-{G} \ca{B}}
\]
 in $\Cat$ has objects the quintuples $(A,f,C,g,B)$ where $A \in \ca{A}$, $B \in \ca{B}$, $C \in \ca{C}$, and $f \colon FA \rightarrow C$ and $g \colon GB \rightarrow C$ are isomorphisms. The morphisms are given by triples of morphisms in $\ca{A}$, $\ca{B}$, and $\ca{C}$ which are compatible with these isomorphisms. Using this we can check that the functor from the bicategorical pullback to the iso-comma category $\ca{A} \pb{\ca{C}} \ca{B}$ which sends $(A,f,C,g,B)$ to $(A,g^{-1} \circ f,B)$ is an equivalence of categories. This shows that $\ca{A}\pb{\ca{C}} \ca{B}$ is also a bicategorical pullback of the above diagram.

 In categories of sets equipped with algebraic structure (groups, rings, modules, etc.), limits can always be constructed as follows: one first forms the limit of the diagram of underlying sets and then uses the universal property to endow this limit with a unique compatible algebraic structure. The same works for bicategorical limits in 2-categories of categories endowed with ``algebraic'' structure, for example, $R$-linear categories, $R$-linear categories with colimits of certain shapes, and symmetric monoidal categories. In each of these cases, the required structure can be written down quite explicitly. A more economical way is to systematically study this phenomenon of universal 2-dimensional algebra. This is known as 2-dimensional monad theory, which has been extensively studied (see \cite{BLACKWELL_KELLY_POWER}, \cite[\S 4]{LACK_COMPANION}, and the many references in \cite[\S 4.4]{LACK_COMPANION}).

\subsection{Tensor categories and base change}\label{section:background_tensor}
 A \emph{tensor category} over $R$ is a cocomplete\footnote{A category is called cocomplete if it has all small colimits.} symmetric monoidal closed $R$-linear category $(\ca{C},\otimes, \U)$. It is called \emph{locally finitely presentable} (lfp for short) if it has a strong generating set consisting of finitely presentable objects (that is, objects $A \in \ca{C}$ such that $\ca{C}(A,-) \colon \ca{C} \rightarrow \Mod_R$ preserves filtered colimits) and finitely presentable objects are closed under finite tensor products. This means in particular that the unit object $\U$ is finitely presentable.

 A \emph{tensor functor} $F \colon \ca{C} \rightarrow \ca{D}$ is a symmetric strong monoidal $R$-linear left adjoint functor. Note that to each lfp tensor category $\ca{C}$ we can associate the right exact symmetric monoidal category $\ca{C}_{\fp}$ of finitely presentable objects in $\ca{C}$, and conversely, to each $\ca{A} \in \ca{RM}$ we can associate the lfp tensor category $\Ind(\ca{A})$. However, while the latter is functorial, the former is in general not: the reason is that tensor functors need not preserve finitely presentable objects. This problem does not arise if $\ca{C}$ is generated by objects with duals. Indeed, any tensor functor preserves duals, and in an lfp tensor category, every object with a dual is finitely presentable.

 The 2-category of weakly Tannakian categories is therefore equivalent to the full sub-2-category of tensor categories consisting of \emph{geometric} tensor categories, that is, tensor categories $\ca{C}$ which are abelian, generated by duals, and which admit a faithful and exact tensor functor
\[
 \ca{C} \rightarrow \Mod_B
\]
 for some commutative $R$-algebra $B$. Since these tensor functors correspond to $\fpqc$-atlases $\Spec(B) \rightarrow X$ of the corresponding Adams stack, we call such tensor functors \emph{coverings}.

 The advantage of working with lfp tensor categories is the large supply of commutative algebra objects that exist in $\ca{C}=\Ind(\ca{A})$ (but rarely in the subcategory $\ca{A}$ of finitely presentable objects).

 Given a commutative algebra $B \in \ca{C}$, we denote the category of $B$-modules in $\ca{C}$ by $\ca{C}_B$. We write
\[
 (-)_B \colon \ca{C} \rightleftarrows \ca{C}_B \colon U 
\]
 for the base change functor $C_B \defl B \otimes C$ (that is, $C_B$ is the free $B$-module on $C$, and $U$ denotes the forgetful functor which sends a $B$-module to its underlying object in $\ca{C}$). Note that $\ca{C}_B$ is a tensor category since $B$ is commutative and that $(-)_B$ is a tensor functor. The tensor category $\ca{C}_B$ is lfp (respectively abelian) if $\ca{C}$ is lfp (respectively abelian).

 A commutative algebra $B \in \ca{C}$ is called \emph{faithfully flat} if $(-)_B$ is conservative (it reflects isomorphisms) and it preserves finite limits.

 The two central concepts used in the definition of universal weakly Tannakian categories are that of a locally free object of constant rank and that of a locally split epimorphism.

\begin{dfn}\label{dfn:locally_free_locally_split_summary}
 Let $\ca{C}$ be a tensor category over $R$. An object $V \in \ca{C}$ is called \emph{locally free of rank $d \in \mathbb{N}$} if there exists a faithfully flat algebra $B \in \ca{C}$ such that there is an isomorphism $V_B \cong B^{\oplus d}$ in the category of $B$-modules. We denote the category of locally free objects of rank $d$ by $\LF^{\rk d}$ and the groupoid of such objects and isomorphisms between them by $\LFdi(\ca{C})$. The full subcategory of locally free objects of finite rank will be denoted by $\LF^c_{\ca{C}}$.

 A morphism $p \colon V \rightarrow W$ in $\ca{C}$ is called a \emph{locally split epimorphism} if there exists a faithfully flat algebra $B \in \ca{C}$ such that $p_B \colon V_B \rightarrow W_B$ is a split epimorphism of $B$-modules.

 A right exact sequence
\[
 \xymatrix{U \ar[r]^-{p} & V \ar[r]^-{q} & W \ar[r] & 0}
\]
 in $\ca{C}$ is called \emph{locally split right exact} if the kernel $K$ of $q$ exists and both $q$ and the induced morphism $U \rightarrow K$ are locally split epimorphisms.
\end{dfn}

 If $\ca{A}$ is a right exact symmetric monoidal category, we call an object (respectively a morphism, a cokernel diagram) \emph{locally free} (respectively \emph{locally split epimorphisms}, or locally split right exact sequence) if its image under the Yoneda embedding is such in the tensor category $\Ind(\ca{A})$. With this convention, the Yoneda embedding induces an equivalence $\LF^c_{\ca{A}} \simeq \LF^c_{\Ind(\ca{A})}$ (see \cite[Remark~5.6.2]{SCHAEPPI_COLIMITS}).

 One of the main technical tools introduced in \cite{SCHAEPPI_COLIMITS} is that of an \emph{Adams algebra}. These are faithfully flat algebras that satisfy very convenient stability properties: they are for example preserved by any tensor functor between lfp tensor categories. Recall that we say that a tensor category $\ca{C}$ has \emph{exact} filtered colimits if filtered colimits in $\ca{C}$ commute with finite limits. Using the fact that this is the case in $\Mod_R$ one can show that any lfp tensor category has exact filtered colimits.

 A commutative algebra $A$ in a tensor category $\ca{C}$ with exact filtered colimits is called an \emph{Adams algebra} if the unit $\eta \colon \U \rightarrow A$ can be written as filtered colimit of morphisms $\eta_i \colon \U \rightarrow A_i$ where each $A_i$ has a dual and each diagram
\[
 \xymatrix{A_i^{\vee} \otimes A_i^{\vee} \ar@<0.5ex>[r]^-{A_i^{\vee} \otimes \eta_i^{\vee}} \ar@<-0.5ex>[r]_-{\eta_i^{\vee} \otimes A_i^{\vee}} & A_i^{\vee} \ar[r]^{\eta_i^{\vee}} & \U }
\]
 is a coequalizer diagram. By \cite[Proposition~4.1.11]{SCHAEPPI_COLIMITS}, Adams algebras in tensor categories with exact filtered colimits are faithfully flat.

 If $\ca{C}$ is also abelian, there is an easier way to characterize Adams algebras: the above diagram is a coequalizer diagram if and only if all the morphisms $\eta_i^{\vee} \colon A_i^{\vee} \rightarrow \U$ are epimorphisms (see \cite[Lemma~4.1.9]{SCHAEPPI_COLIMITS}). We refer the reader to \cite[\S 4.1]{SCHAEPPI_COLIMITS} for basic properties of Adams algebras.

\subsection{Universal weakly Tannakian categories}\label{section:background_universal}

 Let $\ca{C}$ be a lfp tensor category. The \emph{universal geometric tensor category} associated to $\ca{C}$ is defined as follows. Let $\Sigma$ be the class of cokernel diagrams in $\LF^c_{\ca{C}}$ which are locally split right exact sequences in $\ca{C}$. Let
\[
 G(\LF^c_{\ca{C}}) \defl \Lex_{\Sigma}[(\LF^c_{\ca{C}})^{\op},\Mod_R]
\]
 be the category of $R$-linear presheaves $(\LF^c_{\ca{C}})^{\op} \rightarrow \Mod_R$ which send the cokernel diagrams in $\Sigma$ to kernel diagrams in $\Mod_R$.

 The \emph{universal weakly Tannakian category} $T(\ca{A})$ associated to a right exact symmetric monoidal category $\ca{A}$ is the category
\[
 T(\ca{A}) \defl G(\LF^c_{\ca{A}})_{\fp}
\]
 of finitely presentable objects in the universal geoemtric tensor category associated to $\Ind(\ca{A})$. That this is indeed a weakly Tannakian category with the desired universal property is the content of \cite[Theorem~5.6.4]{SCHAEPPI_COLIMITS}.

 From \cite[Proposition~5.5.1]{SCHAEPPI_COLIMITS} we know that there is a canonical equivalence
\[
 \LF^c_{\ca{A}} \rightarrow \LF^c_{T(\ca{A})}
\]
 induced by the Yoneda embedding.

 In order to understand the weakly Tannakian category $T(\ca{A})$, we therefore need two ingredients: we need to understand locally free objects in $\Ind(\ca{A})$, and we need to understand locally split epimorphisms between such in $\Ind(\ca{A})$. In general, this is a difficult problem. However, we mostly care about the case where $\ca{A}$ is obtained as a limit of a diagram of weakly Tannakian categories. As we will see, we can use the fact that such categories $\ca{A}$ are somehow ``close'' to categories of geometric origin to get a good handle on these two concepts.
\section{Recognizing colimits of Adams stacks}\label{section:colimits_of_adams_stacks}

\subsection{Vector bundles on colimits of Adams stacks}
 In order to better understand colimits of Adams stacks, we need a good description of the full subcategory of locally free objects of constant finite rank in $\lim \ca{A}_i$ (where $\ca{A}_i \in \ca{RM}$). We start by describing the \emph{groupoid} of locally free objects of a fixed rank $d \in \mathbb{N}$.

\begin{lemma}\label{lemma:rank_d_groupoid}
 The 2-functor
\[
 \LFdi\bigl(\Ind(-)\bigr) \colon \ca{RM} \rightarrow \Gpd
\]
 which sends a right exact symmetric monoidal category to the groupoid of locally free objects of rank $d$ and isomorphisms between them is represented by $\Rep(\mathrm{GL}_d)_{\fp}$. More precisely, the 2-natural transformation
\[
 \ca{RM}\bigl( \Rep(\mathrm{GL}_d)_{\fp}, \ca{A} \bigr) \rightarrow \LFdi\bigl(\Ind(\ca{A})\bigr)
\]
 which sends $F$ to its value $FV$ at the standard representation $V=\ubar{R}^{d}$ is an equivalence for all $\ca{A} \in \ca{RM}$.
\end{lemma}

\begin{proof}
 Passage to ind-objects gives an equivalence
\[
 \ca{RM}\bigl( \Rep(\mathrm{GL}_d)_{\fp}, \ca{A} \bigr) \rightarrow \Fun_{c,\otimes}\bigl( \Rep(\mathrm{GL}_d), \Ind(\ca{A}) \bigr)
\]
 which is natural in $\ca{A}$ (see \cite[Lemma~2.6]{SCHAEPPI_GEOMETRIC}; this is also a special case of \cite[Theorem~3.2.4]{SCHAEPPI_COLIMITS}, where $\Sigma$ is the class of \emph{all} cokernel diagrams in $\Rep(\mathrm{GL}_d)$).

 The claim follows from the fact that evaluation at the standard representation gives an equivalence
\[
 \Fun_{c,\otimes}\bigl( \Rep(\mathrm{GL}_d),\Ind(\ca{A}) \bigr) \rightarrow \LFdi\bigl( \Ind(\ca{A}) \bigr)
\]
 (see \cite[Theorem~4.3.10]{SCHAEPPI_COLIMITS}).
\end{proof}

 The following lemma is a well known fact in geometric tensor categories. We need to prove it in greater generality since $\Ind(\ca{A})$ will often not be geometric (or even just abelian) for $\ca{A} \in \ca{RM}$. 

\begin{lemma}\label{lemma:rank_of_locally_free_unique}
 Let $\ca{C}$ be a tensor category with $\U \neq 0$ and let $M$, $N \in \ca{C}$ be locally free objects of rank $m$ and $n \in \mathbb{N}$ respectively. If there exists an isomorphism $M \rightarrow N$ in $\ca{C}$, then $m=n$.
\end{lemma}

\begin{proof}
 From the definition of locally free objects it follows that we only need to check that the existence of an isomorphism $\U^m \cong \U^n$ implies that $m=n$. But any such isomorphism lies in the image of a tensor functor
\[
 \Mod_{\End(\U)} \rightarrow \ca{C} \smash{\rlap{,}}
\]
 so the claim follows from the fact that the rank of a free module over a non-zero commutative ring is invariant under isomorphisms (and $\End(\U)$ is non-zero since we assumed that $\U \neq 0$).
\end{proof}

 The 2-functor we care about only preserves certain types of limits.

\begin{dfn}\label{dfn:connected_category}
 A category $\ca{I}$ is called \emph{connected} if any pair of objects $i, j \in \ca{I}$ is connected by a finite zig-zag
\[
 \xymatrix{i  & \ar[l] \cdot \ar[r] & \ldots   & \ar[l] j}
\]
 of morphisms in $\ca{I}$.
\end{dfn}

 For example, the indexing diagrams for pullbacks is connected, while the diagram for binary products is not connected.

\begin{prop}\label{prop:locally_free_in_limit}
 Let $\ca{I}$ be a connected category and let $F \colon \ca{I} \rightarrow \ca{RM}$, $i \mapsto \ca{A}_i$ be a diagram in $\ca{RM}$ such that $\ca{A}_i$ is non-trivial (that is, the unit object of $\ca{A}_i$ is not isomorphic to $0 \in \ca{A}_i$). Suppose that the morphisms $F_i \colon \ca{A} \rightarrow \ca{A}_i$ and the invertible 2-cells
\[
 \xymatrix{ & \ca{A}_i \ar[rd]^{F_{\psi}} \\ \ca{A} \ar[ru]^{F_i} \ar[rr]_{F_j} \rrtwocell\omit{<-3>\quad\varphi_{\psi}} && \ca{A}_j}
\]
 (for each morphism $\psi \colon i \rightarrow j$ in $\ca{I}$) exhibit $\ca{A}$ as bicategorical limit of the $\ca{A}_i$ in $\ca{RM}$. Then $M \in \ca{A}$ is a locally free object of constant finite rank if and only if $F_i(M)$ is a locally free object of constant finite rank in $\ca{A}_i$ for all $i \in \ca{I}$.
\end{prop}

\begin{proof}
 If $M$ is locally free of rank $d$ in $\ca{A}$, the so is $F_i(M)$ (see Lemma~\ref{lemma:rank_d_groupoid}). It remains to check the converse.

 Assume that $F_i(M)$ is locally free of rank $d_i \in \mathbb{N}$ for all $i \in \ca{I}$. Since the category $\ca{I}$ is connected and ranks are invariant under isomorphism in non-trivial tensor categories (see Lemma~\ref{lemma:rank_of_locally_free_unique}), it follows that $F_i (M)$ is locally free of rank $d$ for a fixed $d \in \mathbb{N}$ (independent of $i \in \ca{I}$).

 From Lemma~\ref{lemma:rank_d_groupoid} we know that there are morphisms
\[
 G_i \colon \Rep(\mathrm{GL}_d)_{\fp} \rightarrow \ca{A}_i
\]
 in $\ca{RM}$ such that the value $G_i(V)$ of the standard representation $V$ is isomorphic to $F_i(M)$. Moreover, the same lemma implies that the isomorphisms 
\[
(\varphi_{\psi})_M \colon F_{\psi} \circ F_i (M) \rightarrow F_j (M) 
\]
 are induced by a unique invertible 2-cell
\[
 \xymatrix{ & \ca{A}_i \ar[rd]^{F_{\psi}} \\ \Rep(\mathrm{GL}_d)_{\fp} \ar[ru]^{G_i} \ar[rr]_-{G_j} \rrtwocell\omit{<-3>\quad\gamma_{\psi}} && \ca{A}_j} 
\]
 in $\ca{RM}$. From the uniqueness of this 2-cell it follows that the $G_i$ together with the $\gamma_{\psi}$ form a bicategorical cone on the diagram $F \colon \ca{I} \rightarrow \ca{RM}$. Thus there exists an essentially unique morphism $G \colon \Rep(\mathrm{GL}_d)_{\fp} \rightarrow \ca{A}$ with $G_i \cong F_i G$. Since limits in $\ca{RM}$ are constructed as in $\Cat$, we have $G(V) \cong M$, so $M$ is locally free of rank $d$ in $\ca{A}$ (see Lemma~\ref{lemma:rank_d_groupoid}).
\end{proof}

\begin{rmk}
 The analogous statement of Proposition~\ref{prop:locally_free_in_limit} is \emph{not} true if the category $\ca{I}$ is not connected, or if some of the $\ca{A}_i$ are allowed to be trivial. For example, if $\ca{A}=\ca{A}_1 \times \ca{A}_2$ with projections $F_i \colon \ca{A} \rightarrow \ca{A}_i$, then $F_i(M)$ can be free of rank $d_i$ with $d_1 \neq d_2$, so one of the implications of Proposition~\ref{prop:locally_free_in_limit} fails.
\end{rmk}

\begin{cor}\label{cor:lf_preserves_connected_limits}
 The 2-functor
\[
 \LF^c_{(-)} \colon \ca{RM} \rightarrow \SymMonCat_R
\]
 which sends $\ca{A}$ to the full subcategory consisting of locally free objects of constant finite rank preserves bicategorical limits of connected diagrams consisting of non-trival right exact symmetric monoidal categories. Moreover, for arbitrary diagrams, the comparison functor
\begin{equation}\label{eqn:comparison_morphism}
 \LF^c_{\lim \ca{A}_i} \rightarrow \lim \LF^c_{\ca{A}_i} 
\end{equation}
 is fully faithful.
\end{cor}

\begin{proof}
 We first prove the second claim. The forgetful 2-functors
\[
 \xymatrix{ \ca{RM}  \ar[r]^-{U} & \SymMonCat_R \ar[r] & \Cat_R }
\]
 both create bicategorical limits. This implies in particular that limits of morphisms in $\SymMonCat_R$ which are fully faithful are again fully faithful. 

 From Lemma~\ref{lemma:rank_d_groupoid} it follows that $\LF^c_{(-)}$ is a sub-2-functor of $U$, in the sense that there is a 2-natural transformation
\[
 \LF^c_(\ca{A}) \rightarrow U\ca{A}
\]
 given by the inclusion. From what we just noted above it follows that the two horizontal morphisms in the diagram
\[
 \xymatrix{ \LF^c_{\lim \ca{A_i}} \ar[r] \ar[d] & U \lim \ca{A}_i \ar[d]^{\simeq} \\ \lim \LF^c_{\ca{A}_i} \ar[r] & \lim U\ca{A}_i }
\]
 are fully faithful. Thus the comparison morphism \eqref{eqn:comparison_morphism} is fully faithful as well.

 The statement that this functor is essentially surjective if the diagram connected and all its vertices are non-trivial is precisely the conclusion of Proposition~\ref{prop:locally_free_in_limit}.
\end{proof}

\begin{cor}\label{cor:vector_bundles_on_colimits}
 Let $\ca{I}$ be a connected category and let $X \colon \ca{I} \rightarrow \ca{AS}$ be a diagram in the category of Adams stacks such that $X_i \neq \varnothing$ for all $i \in \ca{I}$. Then the comparison functor
\[
 \VB^c(\colim X_i) \rightarrow \lim \VB^c(X_i)  
\]
 of symmetric monoidal $R$-linear categories is an equivalence (where $\colim X_i$ is the bicategorical colimit of the diagram $X$ in the 2-category of Adams stacks).
\end{cor}

\begin{proof}
 Note that for any Adams stack $Y$ we have
\[
 \VB^c(Y)=\LF^c_{\QCoh_{\fp}(Y)}
\]
 since vector bundles of constant rank are precisely the locally free objects of constant finite rank. The pseudfunctor $\VB^c(-)$ therefore factors as the composite
\[
 \xymatrix@C=45pt{\ca{AS}^{\op} \ar[r]^-{\QCoh_{\fp}(-)} & \ca{RM} \ar[r]^-{\LF^c_{(-)}} & \SymMonCat_R} 
\]
 of two pseudofunctors. Note that the latter preserves limits of the given type by Corollary~\ref{cor:lf_preserves_connected_limits}, but the former does not: under the equivalence between $\ca{AS}^{\op}$ and the full sub-2-category of $\ca{RM}$ of weakly Tannakian categories, the comparison morphism between the image of the limit and the limit of the image is precisely the coreflection $T\bigr(\lim \QCoh_{\fp}(X_i)\bigl) \rightarrow \lim \QCoh_{\fp} (X_i)$. 

 But this means that it suffices to check that $\LF^c(-)$ sends this coreflection to an equivalence of categories, which follows from \cite[Proposition~5.5.1]{SCHAEPPI_COLIMITS} and the definition of $T(-)$.
\end{proof}

 \subsection{The recognition theorem for colimits}
 Our next goal is to prove a theorem which characterizes certain colimit diagrams in terms of their categories of quasi-coherent sheaves. We call a set of objects $\ca{I}_0$ of a small indexing category $\ca{I}$ \emph{isomorphism detecting} if for any diagram $\ca{I} \rightarrow \CAT_R$, $i \mapsto \ca{C}_i$ in the 2-category of (possibly large) $R$-linear categories, the functor
\[
 \xymatrix{\lim \ca{C}_i \ar[r] & \textstyle\prod\nolimits_{i \in \ca{I}_0} \ca{C}_i}
\]
 induced by the canonical projections detects isomorphisms. For example, in the indexing category 
\[
 \xymatrix {0 \ar[r] & 1  & 2 \ar[l] }
\]
 for pullbacks, the set $\{0,2\}$ is isomorphism detecting: a morphism in the bicategorical pullback $\ca{A} \pb{\ca{C}} \ca{B}$ is an isomorphism if and only if its projections to $\ca{A}$ and $\ca{B}$ are isomorphisms.

 Note that $\ca{I}_0=\ca{I}$ is always ismorphism detecting since a morphism in $\lim \ca{C}_i$ is an isomorphism if and only if its component in $\ca{C}_i$ is an isomorphism for all $i \in \ca{I}$.

\begin{thm}\label{thm:colimit_recognition}
 Let $\ca{I}$ be a connected category and let $\ca{I}_0 \subseteq \ca{I}$ be isomorphism detecting as a subset of $\ca{I}^{\op}$. Let $X \colon \ca{I} \rightarrow \ca{AS}$ be a diagram with $X_i \neq \varnothing$ and let
 \[
 \xymatrix{& X_j \ar[rd]^{f_j} \\ X_i \rrtwocell\omit{<-3>\quad\kappa_{\varphi}}   \ar[ru]^{X_{\varphi}} \ar[rr]_-{f_i} && W}
\]
 be a pseudococone\footnote{A pseudococone is the dual notion of a pseudocone introduced in \S \ref{section:background_limits}.} on the diagram $X$. Suppose that the following conditions hold:
\begin{enumerate}
 \item[(i)] the pseudocone obtained by applying $\VB^{c}(-)$ to the pseudococone $(f_i,\kappa_{\varphi})$ induces an equivalence
\[
 \VB^c(W) \rightarrow \lim \VB^c(X_i)
\]
 of $R$-linear categories;

\item[(ii)] the functors 
\[
f_i^{\ast} \colon \QCoh_{\fp}(W) \rightarrow \QCoh_{\fp}(X_i) 
\]
 for $i \in \ca{I}_0$ jointly detect epimorphisms;
\end{enumerate}
 then $(f_i,\kappa_{\varphi})$ exhibits $W$ as bicategorical colimit of the diagram $X$. If, in addition, $\ca{I}$ is finite, the converse is true as well: if $(f_i,\kappa_{\varphi})$ exhibits $W$ as bicategorical colimit in $\ca{AS}$, then the above two conditions are satisfied.
\end{thm}

 That the diagram in question is a colimit diagram under the given assumptions follows from the ``generator and relations'' presentation of the category of quasi-coherent sheaves given in \cite[Corollary~3.3.10]{SCHAEPPI_COLIMITS}, the Tannaka duality results recalled in \S \ref{section:background_tannaka}, and some facts that we already established about limits of weakly Tannakian categories. It is much harder to check the converse, especially that Condition~(ii) is satisfied for the colimit pseudococone. This is perhaps not so surprising since our construction of colimits was very indirect. The proof requires a careful analysis of locally split epimorphisms in (potentially non-abelian) lfp tensor categories.

 In order to do this we will use \cite[Lemma~5.3.11]{SCHAEPPI_COLIMITS}, which states that if $\ca{C}$ is a tensor category with exact filtered colimits, $p \colon M \rightarrow \U$ a morphism whose domain is a retract of a locally free object of constant finite rank, and $F \colon \ca{C} \rightarrow \ca{D}$ is a conservative tensor functor such that $Fp$ is a locally split epimorphism, then  $p$ itself locally split as well. Moreover, under these assumptions $p$ is exhibited as locally split epimorphism by an Adams algebra in $\ca{C}$.

\begin{prop}\label{prop:general_locally_split_criteria}
 Let $\ca{C}$ and $\ca{D}$ be tensor categories with exact filtered colimits and with equalizers, and let $F \colon \ca{C} \rightarrow \ca{D}$ be a conservative tensor functor. Let $p \colon M \rightarrow N$ be a morphism between locally free objects of constant finite rank. Then $p$ is locally split if and only if $Fp$ is locally split.

 Moreover, in this case there exists finite tensor product $A=\bigotimes_{i=1}^n A_i$ of Adams algebras in $\ca{C}$ such that $p_A$ is a split epimorphism.
\end{prop}

\begin{proof}
 For any Adams algebra $A \in \ca{C}$, the induced tensor functor $\overline{F} \colon \ca{C}_A \rightarrow \ca{D}_{FA}$ is also conservative. Since filtered colimits in $\ca{C}$ and $\ca{D}$ are exact, both $A$ and $FA$ are are faithfully flat (see \cite[Proposition~4.1.12]{SCHAEPPI_COLIMITS}). Moreover, $p$ (respectively $Fp$) is a locally split epimorphism if and only if $p_A$ (respectively $(Fp)_{FA}$) is locally split: this follows since any base change functor $(-)_A$ preserves \emph{all} faithfully flat algebras (not just Adams algebras).

 For this reason, we can simplify the situation as follows. By \cite[Theorem~4.3.10]{SCHAEPPI_COLIMITS}, there exists an Adams algebra $A \in \ca{C}$ such that $N_A \cong A^d$ (since this is true for the standard representation in the category $\Rep(\mathrm{GL}_d)$). Passing from $F$ to the tensor functor $\overline{F}$ defined above if necessary, we can therefore assume without loss of generality that $N=\U^d$. We will prove the slightly stronger statement that the claim is true if the domain $M$ is merely a retract of a locally free object of constant finite rank. As we will see, this allows us to proceed by induction.

 Suppose now that $Fp$ is a locally split epimorphism. By composing with a base change functor $(-)_B \colon \ca{D} \rightarrow \ca{D}_B$ if necessary we can assume that $Fp$ is a split epimorphism. Then the composite $F\pi_1 \circ Fp$ is split as well, where $\pi_1 \colon \U^d \rightarrow \U$ denotes the first projection. Since $F$ is conservative, it reflects all the colimits it preserves. From \cite[Lemma~5.3.11]{SCHAEPPI_COLIMITS} it follows that there exists an Adams algebra $A \in \ca{C}$ which splits $\pi_1 p$. By passing to $\overline{F}$ again if necessary, we can assume that there exists a section $s \colon \U \rightarrow M$ of $\pi_1 p$.

 By splitting the idempotent $s \pi_1 p$ on the object $M$, we find that $p$ is---up to isomorphism---given by an upper triangular matrix
\[
 \xymatrix{\U \oplus M^{\prime} 
\ar[r]^{
\bigl( \begin{smallmatrix} 
  1 & \ast \\ 
  0 & p^{\prime}
\end{smallmatrix} \bigr)} 
& \U \oplus \U^{d-1}}
\]
 for some morphism $p^{\prime} \colon M^{\prime} \rightarrow \U$. Note that $M^{\prime}$ is a retract of a locally free object of constant finite rank, and $Fp^{\prime}$ is split since $Fp$ is. Proceeding inductively we reduce to the case $d=1$, where the claim that $p$ is locally split follows from \cite[Lemma~5.3.11]{SCHAEPPI_COLIMITS}. Note that we have in fact shown that $p$ is split by a finite tensor product of Adams algebras in $\ca{C}$.

 It remains to show the converse: if $p \colon M \rightarrow \U^d$ is a locally split epimorphism (split by some faithfully flat algebra $B$, say), then $Fp$ is a locally split epimorphism. Note that $B$ might not itself be an Adams algebra, so we do not know whether or not $FB$ is faithfully flat.

 However, if we apply the direction we proved above to the conservative tensor functor $(-)_B \colon \ca{C} \rightarrow \ca{C}_B$, we find that there exists a finite tensor product $A=\bigotimes_{i=1}^n A_i$ of Adams algebras which exhibit $p$ as locally split epimorphism. The fact that $F$ preserves Adams algebras implies that $Fp$ is a locally split epimorphism.
\end{proof}

\begin{lemma}\label{lemma:ff_on_fp_implies_conservative}
 Let $\ca{C}$ be a lfp $R$-linear category, $\ca{D}$ a cocomplete $R$-linear category, and $L \colon \ca{C} \rightarrow \ca{D}$ a functor which is cocontinuous and preserves finitely presentable objects. If the restriction of $L$ to $\ca{C}_{\fp}$ is fully faithful, then $F$ is conservative.
\end{lemma}

\begin{proof}
 We first show that $L$ is faithful. Let $f \colon X \rightarrow Y$ be a morphism in $\ca{C}$ such that $Lf=0$. Write $Y$ as a filtered colimit $Y \cong \colim Y_i$ of finitely presentable objects $Y_i$. It suffices to check that for every $k \colon A \rightarrow X$ with $A$ finitely presentable, we have $fk=0$.

 Since $A$ is finitely presentable, there exists a morphism $f^{\prime}$ such that the diagram
\[
 \xymatrix{A \ar[r]^-{k} \ar[d]_{f^{\prime}} & X \ar[d]^{f} \\ Y_i \ar[r]^-{\lambda_i} & Y}
\]
 is commutative, where $\lambda_i$ denotes the structure morphism of the colimit. By assumption, $LA$ is finitely presentable and $LY \cong \colim LY_i$, so the fact that $L\lambda_i L f^{\prime}=0$ implies that $L$ sends the composite
\[
 \xymatrix{A \ar[r]^-{f^{\prime}} & Y_i \ar[r] & Y_j }
\]
 to the zero morphism for some $j$ in the indexing category. But both $A$ and $Y_j$ lie in $\ca{C}_{\fp}$, so the above composite---and therefore $fk$---must be equal to the zero morphism, which concludes the proof that $L$ is faithful.

 Now let $f \colon X \rightarrow Y$ be a morphism such that $Lf$ is an isomorphism. Since $\ca{C}_{\fp}$ is a strong generator (it is even dense), it suffices to check that for any $k \colon A \rightarrow Y$ with $A \in \ca{C}_{\fp}$, there exists a unique $k^{\prime} \colon A \rightarrow X$ with $fk^{\prime}=k$. Since we already established that $L$ is faithful, it suffices to check existence of such a morphism $k^{\prime}$.

 Write $X$ as filtered colimit $X \cong \colim X_i$ with $X_i \in \ca{C}_{\fp}$. By assumption, $LA$ is finitely presentable, so the composite $(Lf)^{-1} \circ Lk \colon LA \rightarrow LX \cong \colim LX_i$ factors through one of the structure morphisms $L\lambda_i \colon LX_i \rightarrow LX$ of the colimit. Since both $A$ and $X_i$ lie in $\ca{C}_{\fp}$, the assumption on $L$ implies that this factorization is of the form $L \ell$ for some morphism $\ell \colon A \rightarrow X_i$. Since $L$ is faithful, we find that the morphism $k^{\prime}=\lambda_i \ell$ gives the desired lift of $k$.
\end{proof}

 Recall that we write $\Rex$ for the 2-category of finitely cocomplete $R$-linear categories and right exact functors (that is, finite colimit preserving functors) between them.

\begin{lemma}\label{lemma:ind_lim_to_prod_ind_conservative}
 Let $\ca{I}$ be a finite category with isomorphism detecting set $\ca{I}_0 \subseteq \ca{I}$ and let $\ca{I} \rightarrow \Rex$, $i \mapsto \ca{A}_i$ be a diagram. Then the left adjoint
\[
F \colon \Ind(\lim \ca{A}_i) \rightarrow \textstyle\prod\nolimits_{i \in \ca{I}_0} \Ind(\ca{A}_i)
\]
 induced by the structure morphisms $\lim \ca{A}_i \rightarrow \ca{A}_i$ is conservative.
\end{lemma}

\begin{proof}
 Let $\ca{C}_i =\Ind(\ca{A}_i)$. The functor in question factors as the composite of the functor $L \colon \Ind(\lim \ca{A}_i) \rightarrow \lim(\ca{C}_i)$, followed by the functor induced by the canonical projections $\lim \ca{C}_i \rightarrow \prod_{i \in \ca{I}_0} \ca{C}_i$. The latter is conservative since this limit is computed as in the 2-category $\CAT_R$ of (possibly large) $R$-linear categories (this is precisely the definition of an isomorphism detecting set $\ca{I}_0$).

 We claim that $L$ satisfies the conditions of Lemma~\ref{lemma:ff_on_fp_implies_conservative} above. Since limits in $\Rex$ are computed as in $\Cat_R$, the restriction of $L$ to $\Ind(\lim \ca{A}_i)_{\fp} \simeq \lim \ca{A}_i$ is fully faithful. It preserves colimits since it is built from cocontinuous functors. It only remains to check that it also preserves finitely presentable objects.

 This boils down to the statement that an object $(C_i,\kappa_{\varphi})$ of $\lim \ca{C}_i$ is finitely presentable if each $C_i$ is finitely presentable in $\ca{C}_i$. We need to check that any morphism from $(C_i,\kappa_{\varphi})$ to a filtered colimit $\colim_j \bigl(X_i(j),\kappa_{\varphi}(j) \bigl)$ factors through one of the stages, and that any two such factorizations become equal at a higher stage. 

 Since colimits in $\lim \ca{C}_i$ are computed levelwise, we find that for each $i$, the morphism $C_i \rightarrow \colim_j X_i(j)$ factors through one of the colimit stages $X_i(j)$. A priori, these factorizations might not constitute a morphism in $\lim \ca{C}_i$. But since the indexing category $\ca{I}$ is finite, we only need to check that a finite number of diagrams commute to establish this. Therefore we can achieve this by passing to a (common) higher stage $j^{\prime}$. Moreover, any two such factorizations become equal after passage to a possibly higher stage (we can find such a stage for each $\ca{C}_i$, and the claim again follows since $\ca{I}$ is finite).

 This shows that $L$ does indeed preserve finitely presentable objects. If follows from Lemma~\ref{lemma:ff_on_fp_implies_conservative} that $L$ is conservative.
\end{proof}

\begin{lemma}\label{lemma:structural_morphisms_detect_epis}
 Let $\ca{I}$ be a finite category with isomorphism detecting set $\ca{I}_0 \subseteq \ca{I}$. Let $\ca{I} \rightarrow \ca{RM}$, $i \mapsto \ca{A}_i$ be a diagram such that each $\ca{A}_i$ is a weakly Tannakian category. Then the functor
\[
 T( \lim \ca{A}_i) \rightarrow \textstyle \prod \nolimits_{i \in \ca{I}_0} \ca{A}_i
\]
 detects epimorphisms.
\end{lemma}

\begin{proof}
 From the definition of universal weakly Tannakian categories (see \S \ref{section:background_universal} and \cite[Definition~5.6.3]{SCHAEPPI_COLIMITS}) and the fact that the universal weakly Tannakian category gives a coreflection (see \cite[Theorem~5.6.4]{SCHAEPPI_COLIMITS}) it follows that the statement is equivalent to the claim that the composite
\[
 \xymatrix{ G(\LF^c_{\lim \ca{A}_i}) \ar[r]^-{E} & \Ind(\lim \ca{A}_i) \ar[r]^{F} & \textstyle\prod\nolimits_{i \in \ca{I}} \ca{A}_i }
\]
 (where $E$ denotes the functor from \cite[Remark~5.4.5]{SCHAEPPI_COLIMITS} and $F$ denotes the functor from Lemma~\ref{lemma:ind_lim_to_prod_ind_conservative}) detects epimorphisms between finitely presentable objects.

 We first claim that the composite detects epimorphisms between objects of the generating subcategory $\LF^c_{\lim \ca{A}_i}$ of $G(\LF^c_{\lim \ca{A}_i})$. By \cite[Proposition~5.4.13]{SCHAEPPI_COLIMITS}, this boils down to checking that a morphism $f \colon A \rightarrow B$ between locally free objects of constant finite rank in the (possibly non-abelian category) $\Ind(\lim \ca{A}_i)$ is a locally split epimorphism if $Ff$ is an epimorphism.

 Thus assume that $Ff$ is an epimorphism. Since a finite product of geometric tensor categories is clearly geometric, $Ff$ is an epimorphism between locally free objects in a geometric tensor category. It follows from \cite[Example~5.2.2]{SCHAEPPI_COLIMITS} that $Ff$ is locally split. Since $F$ is conservative (see Lemma~\ref{lemma:ind_lim_to_prod_ind_conservative}), it detects locally split epimorphisms (see Proposition~\ref{prop:general_locally_split_criteria}). This shows that the tensor functor $F \circ E$ does indeed detect epimorphisms between objects in $\LF^c_{\lim \ca{A}_i}$.

 The claim therefore follows from Lemma~\ref{lemma:epi_detection_in_lfp_abelian_cats} below (which is applicable since any tensor functor between geometric tensor categories preserves finitely presentable objects).
\end{proof}

\begin{lemma}\label{lemma:epi_detection_in_lfp_abelian_cats}
 Let $\ca{C}$ and $\ca{D}$ be lfp abelian categories, and let $\ca{A} \subseteq \ca{C}_{\fp}$ be a generator which is closed under finite direct sums. Suppose that $L \colon \ca{D} \rightarrow \ca{D}$ is a cocontinuous functor which preserves finitely presentable objects. Then the following conditions are equivalent
\begin{enumerate}
 \item[(i)] The functor $L$ detects epimorphisms in $\ca{C}$ whose domain and codomain lie in $\ca{A}$;
 \item[(ii)] If $C \in \ca{C}$ is a finitely generated object (that is, an object admitting an epimorphism $C_0 \rightarrow C$ in $\ca{C}$ for some $C_0 \in \ca{C}_{\fp}$), then $LC \cong 0$ implies $C \cong 0$;
\item[(iii)] If $C_0 \in \ca{C}_{\fp}$ and $LC_0 \cong 0$, then $C_0 \cong 0$.
\item[(iv)] The functor $L$ detects epimorphisms in $\ca{C}$ whose domain and codomain lie in $\ca{C}_{\fp}$.
\end{enumerate}
\end{lemma}

\begin{proof}
 The implications $(ii) \Rightarrow (iii)$ and $(iv) \Rightarrow (i)$ are immediate. To see that $(i) \Rightarrow (ii)$, let $C$ be a finitely generated object such that $LC \cong 0$. Since $\ca{A}$ is a generator which is closed under finite direct sums, we can find an object $A \in \ca{A}$ and an epimorphism $f \colon A \rightarrow C$ in $\ca{C}$. We need to show that $C\cong 0$. To see this, take any right exact sequence
\[
 \xymatrix{\textstyle \bigoplus\nolimits_{i \in I} A_i \ar[r] & A \ar[r]^-{f} & C \ar[r] & 0 }
\]
 where $A_i \in \ca{A}$ and $I$ is a (possibly infinite) indexing set. Then $L$ sends this to the right exact sequence
\[
 \xymatrix{ \textstyle \bigoplus\nolimits_{i \in I} LA_i \ar[r] & LA \ar[r] & 0}
\]
 in $\ca{D}$. Since $LA$ is finitely presentable and $\ca{D}$ is a lfp abelian category, there exists a finite set $J \subseteq I$ such that $\bigoplus_{j \in J } LA_j \rightarrow LA$ is an epimorphim. By assumption on $L$, this implies that $\bigoplus_{j \in J } A_j \rightarrow A$ (and therefore $\bigoplus_{i \in I} A_i \rightarrow A$) is an epimorphism in $\ca{C}$. It follows that $C\cong 0$, as claimed.

 It only remains to check that $(iii) \Rightarrow (iv)$. Thus let $f \colon X \rightarrow Y$ be a morphism in $\ca{C}_{\fp}$ such that $Lf$ is an epimorphism in $\ca{D}$. Let $K$ be the cokernel of $f$. Then $K$ is finitely presentable since $\ca{C}_{\fp}$ is closed under finite colimits and $LK \cong 0$, so it follows that $K \cong 0$. Thus $f$ is indeed an epimorphism, as claimed.
\end{proof}

 With this in hand, we can now prove Theorem~\ref{thm:colimit_recognition}.

\begin{proof}[Proof of Theorem~\ref{thm:colimit_recognition}]
 We first show the second statement. The comparison functor
\[
 \VB^{c}(W) \rightarrow \lim \VB^{c}(X_i)
\]
 is an equivalence by Corollary~\ref{cor:vector_bundles_on_colimits} and the functor
\[
 \QCoh_{\fp}(W) \rightarrow \textstyle \prod_{i \in \ca{I}_0} \QCoh_{\fp}(X_i)
\]
 detects epimorphisms by Lemma~\ref{lemma:structural_morphisms_detect_epis}.

 It remains to check the converse: if the cocone $(f_i, \kappa_\varphi)$ induces an equivalence with the limit after applying $\VB^c(-)$ and if the $f_i^{\ast} \colon \QCoh_{\fp}(W) \rightarrow \QCoh_{\fp}(X_i)$ jointly detect epimorphisms, then the induced morphism $f \colon \colim X_i \rightarrow W$ is an equivalence of Adams stacks. By generalized Tannaka duality it suffices to check that the corresponding functor
\[
 f^{\ast} \colon \QCoh(W) \rightarrow \QCoh(\colim X_i)
\]
 is an equivalence. 

 Assumption~(i) and Corollary~\ref{cor:vector_bundles_on_colimits} imply that the restriction of $f^{\ast}$ induces an equivalence $\VB^c(W) \rightarrow \VB^c(\colim X_i)$.

 Let $\Theta$ be the class of morphisms in $\VB^c(W)$ which are epimorphisms in the category $\QCoh(W)$ (equivalently, in $\QCoh_{\fp}(W)$), and let $\Theta^{\prime}$ be the class of morphisms in $\VB^c(\colim X_i)$ which are epimorphisms in $\QCoh(\colim X_i)$. Since $f^{\ast}$ is a left adjoint, we have $f^{\ast}(\Theta) \subseteq \Theta^{\prime}$. Conversely, if $p$ is a morphism in $\VB^{c}(W)$ such that $f^{\ast}(p) \in \Theta^{\prime}$, then (by compsing with the inverse image functor of the structure morphism $X_i \rightarrow \colim X_i$) we find that $f_i^{\ast}(p)$ is an epimorphism for all $i \in \ca{I}_0$. Assumption~(ii) therefore implies that $p$ is an epimorphism in $\QCoh_{\fp}(W)$, that is, $p \in \Theta$.

 The claim that $f^{\ast} \colon \QCoh(W) \rightarrow \QCoh(\colim X_i)$ is an equivalence now follows from \cite[Corollary~3.4.3]{SCHAEPPI_COLIMITS}. 
\end{proof}

 Theorem~\ref{thm:pushout_recognition} is a special case of Theorem~\ref{thm:colimit_recognition}.

\begin{proof}[Proof of Theorem~\ref{thm:pushout_recognition}]
 The indexing category for pushouts
\[
 \xymatrix{0 & \ar[l] 1 \ar[r]& 2}
\]
 is finite and connected. As mentioned at the beginnig of this section, the set $\{0,2\}$ is an isomorphism-detecting set in its opposite category (the indexing category for pullbacks). Theorem~\ref{thm:colimit_recognition}, applied to this indexing category, gives the statement of Theorem~\ref{thm:pushout_recognition}: Condition~(ii) of the former is equivalent to Condition~(ii) of the latter by Lemma~\ref{lemma:epi_detection_in_lfp_abelian_cats}. 
\end{proof}

 We can also use the above lemmas to simplify the description of the category of quasi-coherent sheaves on a finite colimit of Adams stacks with connected indexing category. Often one of the least tractable parts in the definition of the geometric tensor category $G(\ca{A},\ca{C})$ is the class $\Sigma$ of locally split right exact sequences in $\ca{C}$ with entries in $\ca{A}$. The following proposition shows that there is a simpler description of this class $\Sigma$ if the category $\ca{C}$ is ``not too far'' from geometric tensor categories.

\begin{prop}\label{prop:epi_implies_locally_split_criterion}
 Let $\ca{C}$ be a lfp tensor category over $R$. Suppose there exists a family of tensor functors $F_i \colon \ca{C} \rightarrow \ca{C}_i$ which is jointly conservative and such that each $\ca{C}_i$ is geometric. Then each epimorphism $f \colon A \rightarrow A^{\prime}$ in $\ca{C}$ between locally free objects of constant finite rank is a locally split epimorphism. Moreover, each cokernel diagram in $\ca{C}$ consisting of locally free objects of constant finite rank is a locally split right exact sequence.
\end{prop}

\begin{proof}
 The assumption implies that
\[
 (F_i)_{i \in I} \colon \ca{C} \rightarrow \textstyle\prod\nolimits_{i \in I} \ca{C}_i
\]
 is conservative, and while $\prod_{i \in I} \ca{C}_i$ might not be lfp, it does have all limits and exact filtered colimits. From Proposition~\ref{prop:general_locally_split_criteria} it follows that $f$ is a locally split epimorphism if and only if $(F_i f)_{i \in I}$ is a locally split epimorphism in $\prod_{i \in I} \ca{C}_i$, which is the case if each $F_i f$ is a locally split epimorphism. Since $F_i f$ preserves objects with duals, the claim follows from the fact that any epimorphism between objects with duals in a \emph{geometric} tensor category is locally split (see \cite[Example~5.2.2]{SCHAEPPI_COLIMITS}).

 It remains to check that each cokernel diagram
\[
 \xymatrix{L \ar[r]^{p} & M \ar[r]^{q} & N }
\]
 in $\ca{C}$ where $L$, $M$, and $N$ are locally free of constant finite rank is a locally split right exact sequence. By definition of such sequences we need to check that both $q$ and the induced morphism $p^{\prime} \colon L \rightarrow K$ is a locally split epimorphism, where $K$ denotes the kernel of $q$. By \cite[Lemma~5.4.10]{SCHAEPPI_COLIMITS}, $K$ is locally free. By what we just proved, it therefore suffices to check that the comparison morphism $p^{\prime}$ is an epimorphism in $\ca{C}$.

 First note that the sequence
\[
 \xymatrix{0 \ar[r] & F_i K \ar[r] & F_i M \ar[r]^-{F_i q} & F_i N \ar[r] & 0 }
\]
 is exact in $\ca{C}_i$ for each $i \in I$. Indeed, this follows from the fact that there exists a finite tensor product of Adams algebras $A=\bigotimes_{i=1}^n A_i$ such that $q_A$ is a split epimorphism (see Proposition~\ref{prop:epi_implies_locally_split_criterion}): for any such $A$, the above sequence is sent to a split exact sequence by the conservative functor $(-)_{F_i A}$. Since $\ca{C}_i$ is abelian it follows that $F_i p^{\prime} \colon F_i L \rightarrow F_i K$ is an epimorphism.

 The claim that $p^{\prime}$ is an epimorphism in $\ca{C}$ therefore follows from the fact that the cocontinuous functors $F_i$ are jointly conservative.
\end{proof}

Combining this proposition with Lemma~\ref{lemma:ind_lim_to_prod_ind_conservative}, we get the following description of the category of quasi-coherent sheaves on a finite colimit in $\ca{AS}$ with connected indexing type.

\begin{prop}
 Let $\ca{I}$ be a connected finite category and let $\ca{I}_0 \subseteq \ca{I}$ be an isomorphism detecting set of $\ca{I}^{\op}$. Let $X \colon \ca{I} \rightarrow \ca{AS}$ be a diagram with $X_i \neq \varnothing$. Let
\[
\ca{A}=\lim \VB^c(X_i) 
\]
 and write $F_i \colon \ca{A} \rightarrow \VB^c(X_i)$ for the structure morphisms of this limit. Finally, let $\Sigma$ be the class of cokernel diagrams in $\ca{A}$ which are sent to a cokernel diagram in $\QCoh_{\fp}(X_i)$ by each $F_i$, $i \in \ca{I}_0$. Then there is an equivalence
\[
 \QCoh(\colim X_i) \simeq \Lex_{\Sigma}[\ca{A}^{\op},\Mod_R]
\]
 of tensor categories.
\end{prop}

\begin{proof}
 From \cite[Theorem~5.6.4]{SCHAEPPI_COLIMITS} and Proposition~\ref{cor:lf_preserves_connected_limits} we know that there is an equivalence
\[
 \QCoh(\colim X_i) \simeq G(\LF^c_{\ca{A}})= \Lex_{\Sigma_0}[\ca{A}^{\op},\Mod_R]
\]
 of tensor categories, where $\Sigma_0$ denotes the class of cokernel diagrams in $\ca{A}$ that are locally split right exact sequences in $\Ind\bigl(\lim \QCoh_{\fp}(X_i)\bigr)$. This reduces the problem to checking that $\Sigma_0=\Sigma$.

 From Lemma~\ref{lemma:ind_lim_to_prod_ind_conservative} we know that the tensor functors 
\[
F_i^{\ast} \colon \Ind\bigl(\lim \QCoh(X_i) \bigr) \rightarrow \QCoh(X_i) 
\]
 induced by the functors $F_i$ with $i \in \ca{I}_0$ are jointly conservative. From this it follows that $\Sigma$ is precisely the class of cokernel diagrams in $\Ind\bigl(\lim \QCoh(X_i) \bigr)$ with entries in $\ca{A}$ (recall that the $F_i^{\ast}$ are in particular left adjoints, so they jointly detect coequalizer diagrams). But each of these is a locally split right exact sequence by Proposition~\ref{prop:epi_implies_locally_split_criterion}. This shows that $\Sigma \subseteq \Sigma_0$. The converse inclusion is immediate from the fact that the functors $F_i^{\ast}$ induced by the $F_i$ are left adjoints.
\end{proof}

\section{Applications}\label{section:applications}

\subsection{Infinite affine coproducts}

 In this section we compute a few colimit diagrams in the 2-category of Adams stacks explicitly. In all cases, we will start with a diagram of affine schemes and we will show that the colimit in question is again an affine scheme. To do this we will use the following proposition. Recall from \S \ref{section:background_tensor} that we call an object of a right exact symmetric monoidal category $\ca{A}$ \emph{locally free of rank $d \in \mathbb{N}$} (respectively \emph{locally free of constant finite rank}) if it is such considered as an object of $\Ind(\ca{A})$.

\begin{prop}\label{prop:affine_criterion}
 Let $\ca{A}$ be a right exact symmetric monoidal $R$-linear category. Suppose that all epimorphisms in $\ca{A}$ between locally free objects of constant finite rank are split and for all locally free objects $M \in \ca{A}$ of constant finite rank, there exists an epimorphism $\U^{\oplus n} \rightarrow M$ for some $n \in \mathbb{N}$.

 Then the associated weakly Tannakian category $T(\ca{A})$ is \emph{affine}: the canonical right exact symmetric strong monoidal functor
\[
 \Mod_{\End(\U)}^{\fp} \rightarrow T(\ca{A})
\]
 is an equivalence.
\end{prop}

\begin{proof}
 By Definition of $T\ca{A}$, to prove the claim we need to show that the canonical tensor functor
\[
 \Mod_{\End(\U)} \rightarrow G(\LF^c_{\ca{A}})
\]
 is an equivalence. Recall that $G(\LF^c_{\ca{A}})$ is the category $\Lex_{\Sigma}[(\LF^c_{\ca{A}})^{\op},\Mod_R]$, where $\Sigma$ denotes the class of cokernel diagrams in $\LF^{c}_{\ca{A}}$ which are locally split right exact sequences in $\Ind(\ca{A})$. But any such cokernel diagram is by assumption a split right exact sequence. It follows that any $R$-linear presheaf
\[
 F \colon (\LF^c_{\ca{A}})^{\op} \rightarrow \Mod_R
\]
 sends cokernel diagrams in $\Sigma$ to kernel diagrams, that is, we have $G(\LF^c_{\ca{A}})=\mathcal{P}(\LF^{c}_{\ca{A}})$.

 From the second assumption it follows that all objects of $\LF^c_{\ca{A}}$ are direct summands of finite direct sums of the unit object $\U$. Thus the canonical functor
\[
 \Mod_{\End(\U)}=\Prs{\End(\U)} \rightarrow \mathcal{P}(\LF^c_{\ca{A}})
\]
 induced by the inclusion of the full subcategory on the one object $\U$ of $\LF^c_{\ca{A}}$ is an equivalence, as claimed.
\end{proof}

 Using this, we can try to compute infinite coproducts $\coprod \Spec(A_i)$ in the 2-category of Adams stacks. From the examples given by Bhatt (\cite[Examples~8.3 and 8.4]{BHATT}) we know that this coproduct will in general \emph{not} be affine. In both these examples, the dimensions of the $A_i$ do not have a uniform bound (in the second example, $\Spec(A_i)$ has non-trivial \'etale cohomology in a degree that grows with $i$). As we will see below, this is in some sense the only obstruction to $\coprod \Spec(A_i)$ being affine. More precisely, if the minimal number of generators $\mu_i(d)$ of locally free $A_i$-modules of rank $d$ is bounded for each $d$ (that is, $\mu_i(d) \leq \mu(d)$ for some $\mu(d) \in \mathbb{N}$ independent of $i$), then $\coprod_{i \in I} \Spec(A_i) \cong \Spec(\prod_{i \in I} A_i)$ in the 2-category of Adams stacks.

\begin{thm}\label{thm:infinite_affine_coproducts}
 Let $(A_i)_{i \in I}$ be a family of commutative $R$-algebras. If there exist constants $\mu(d) \in \mathbb{N}$ for all $d \in \mathbb{N}$ such that all locally free $A_i$-modules of rank $d$ can be generated by at most $\mu(d)$ elements, then the canonical comparison morphism
\[
 \textstyle\coprod\nolimits_{i \in I} \Spec(A_i) \rightarrow \Spec(\textstyle \prod\nolimits_{i \in I} A_i)
\]
 (where the domain denotes the bicategorical coproduct in the 2-category $\ca{AS}$ of Adams stacks) is an equivalence. In other words, for any Adams stack $X$, the canonical functor
\[
 X(\textstyle\prod\nolimits_{i \in I} A_i) \rightarrow \textstyle\prod\nolimits_{i \in I} X(A_i)
\]
 is an equivalence of groupoids.
\end{thm}

\begin{proof}
 From \cite[Theorem~5.6.4]{SCHAEPPI_COLIMITS} and the equivalence between the 2-category of Adams stacks and the 2-category of weakly Tannakian categories we know that there is an equivalence
\[
 \QCoh_{\fp}\bigl(\textstyle\coprod\nolimits_{i \in I} \Spec(A_i) \bigr) \simeq T(\textstyle\prod\nolimits_{i \in I} \Mod^{\fp}_{A_i})
\]
 of weakly Tannakian categories. 

 The unit object of $\prod\nolimits_{i \in I} \Mod^{\fp}_{A_i}$ has endomorphism ring isomorphic to $\prod\nolimits_{i \in I} A_i$. Thus it suffices to show that the canonical right exact symmetric monoidal $R$-linear functor
\[
 \Mod^{\fp}_{\textstyle\prod\nolimits_{i \in I} A_i} \rightarrow T(\textstyle \prod\nolimits_{i \in I} \Mod^{\fp}_{A_i})
\]
 is an equivalence. This statement is precisely the conclusion of Proposition~\ref{prop:affine_criterion}, so we need to check that the two hypotheses of Proposition~\ref{prop:affine_criterion} are satisfied.

 From Lemma~\ref{lemma:rank_d_groupoid} we know that an object $(M_i)_{i \in I}$ of $\prod \Mod^{\fp}_{A_i}$ is locally free of rank $d$ if and only if all $M_i$ are locally free of rank $d$. In particular, each $M_i$ is projective in this case. Since epimorphisms in the product are precisely the collections of morphisms $(f_i)_{i \in I}$ such that $f_i$ is an epimorphism in $\Mod^{\fp}_{A_i}$, it follows that all epimorphisms between locally free objects of constant finite rank in $\prod \Mod^{\fp}_{A_i}$ are split.

 It remains to check that for all locally free objects $(M_i)_{i \in I}$ of rank $d \in \mathbb{N}$, there exists a natural number $n \in \mathbb{N}$ and an epimorphism $(A_i^n)_{i \in I} \rightarrow (M_i)_{i \in I}$ in $\prod \Mod^{\fp}_{A_i}$. By assumption, such an epimorphism exists if we take $n=\mu(d)$.
\end{proof}

\subsection{Pushouts of affine schemes}
 
 In this section we give two examples of diagrams
\[
 \xymatrix{A \ar[r] \ar[d] & B \ar[d] \\ C \ar[r] & D}
\]
 in the category of commutative $R$-algebras such that the corresponding diagram
\[
 \xymatrix{\Spec(D) \ar[d] \ar[r] & \Spec(C) \ar[d] \\ \Spec(B) \ar[r] & \Spec(A)}
\]
 is a (bicategorical) pushout diagram in the category of Adams stacks. The following example shows that this is not a tautology.

\begin{example}
 Let $A$ be a non-zero commutative $R$-algebra. Then the diagram
\[
 \xymatrix{\Spec(A \times A) \ar[r]^-{\Spec(\Delta)} \ar[d]_{\Spec(\Delta)} & \Spec(A ) \ar[d]^{\id} \\
 \Spec(A) \ar[r]_-{\id} & \Spec(A) }
\]
 is a pushout diagram in the category of all schemes, but it is \emph{not} a (bicategorical) pushout diagram in the 2-category $\ca{AS}$ of Adams stacks. Indeed, the bicategorical pushout $P$ in $\ca{AS}$ has a universal property dual to that of the inertia stack: the groupoid of morphisms $P \rightarrow X$ is equivalent to the groupoid of $A$-points $x \in X(A)$ equipped with an automorphism $\xi \colon x \rightarrow x$. On the other hand, morphisms $\Spec(A) \rightarrow X$ simply correspond to $A$-points of $X$.
\end{example}

 Our first example follows from the Beauville--Laszlo theorem and Theorem~\ref{thm:pushout_recognition}.

\begin{thm}\label{thm:tubular}
 Let $A$ be a commutative $R$-algebra, let $f \in A$ be an element which is not a zero-divisor, let $A_f$ be the localization at $f$, and let $\hat{A} \defl \lim A \slash (f^n)$. Then
\[
\xymatrix{ \Spec(A_f) \pb{\Spec A} \Spec(\hat{A}) \ar[r] \ar[d] & \Spec(\hat{A}) \ar[d] \\ \Spec(A_f)  \ar[r] & \Spec(A)}
\]
 is a pushout diagram in $\ca{AS}$. In other words, for each Adams stack $X$, the canonical functor
\[
 X(A) \rightarrow X(A_f) \pb{X(A_f \ten{A} \hat{A}) } X(\hat{A})
\]
 is an equivalence of groupoids.
\end{thm}

\begin{proof}
 By Theorem~\ref{thm:pushout_recognition}, we need to show that the functor
\begin{equation}\label{eqn:comparison_functor_to_pb}
 \Mod_A \rightarrow \Mod_{A_f} \pb{\Mod_{A_f \ten{A} \hat{A}}} \Mod_{\hat{A}}
\end{equation}
 induces an equivalece when restricted to finitely generated projective modules of constant rank and that the two functors $A_f \ten{A} -$ and $\hat{A} \ten{A} -$ jointly detect epimorphisms between finitely presentable $A$-modules.

 The latter follows from the fact that $A_f \times \hat{A}$ is a faithful $A$-algebra, see \cite[Lemme~1 and 2]{BEAUVILLE_LASZLO}.

 From the Beauville--Laszlo theorem we know that the functor above induces an equivalence when restricted to finitely generated projective modules (see \cite[Remarque~1]{BEAUVILLE_LASZLO}). It only remains to check that $M$ is finitely generated and both $M_f$ and $M_{\hat{A}}$ have constant rank $d$, then $M$ has constant rank $d$ as well.

 If this were not the case, we could write $M=V \oplus W$ where $\rk(W) \neq d$ (at each point of $\Spec(A)$) and $W \neq 0$. It follows that $W_f\cong 0$, that is, $W$ is annihilated by a power of $f$. But since $W$ is projective, this contradicts the fact that $f$ is not a zero-divisor in $A$: since the module $W$ embeds in some $A^n$, multiplication with $f$ on the latter would have non-trivial kernel.
\end{proof}

\begin{lemma}\label{lemma:lifting_sections}
 Let $A$ be a commutative $R$-algebra, $I \subseteq A$ an ideal, $N$ a projective $A$-module, and $p \colon M \rightarrow N$ an epimorphism of $A$-modules. If $\overline{s} \colon N \slash IN \rightarrow M\slash IM$ is a section of the induced morphism $\overline{p} \colon M \slash IM \rightarrow N\slash IN$, then there exists a section $s$ of $p$ such that the diagram
\[
 \xymatrix{N \ar[r]^-{s} \ar[d] & M \ar[d] \\ N \slash IN \ar[r]_-{\overline{s}} & M \slash IM}
\]
 is commutative, where the vertical arrows denote the canonical projections.
\end{lemma}

\begin{proof}
 Since $N$ is projective, there exists \emph{some} splitting of $p$, that is, there exists an $A$-module $M^{\prime}$ and an isomorphism $f \colon N \oplus M^{\prime} \rightarrow M$ such that $pf$ is given by the projection $N\oplus M^{\prime} \rightarrow N$. It follows that $\overline{p} \overline{f} \colon N \slash IN \oplus M^{\prime} \slash IM^{\prime} \rightarrow N\slash IN$ is given by the projection as well. Thus $(\overline{f})^{-1} \circ \overline{s} \colon N \slash IN \rightarrow N \slash IN \oplus M^{\prime} \slash IM^{\prime}$ is given by $(\id, \overline{h})$ for some $\overline{h} \colon N \slash IN \rightarrow M^{\prime} \slash IM^{\prime}$. Again using the fact that $N$ is projective, we can find a morphism $h \colon N \rightarrow M^{\prime}$ such that the diagram
\[
 \xymatrix{N \ar[r]^-{h} \ar[d] & M^{\prime} \ar[d] \\ N \slash IN \ar[r]_-{\overline{h}} & M^{\prime} \slash IM^{\prime} }
\]
 is commutative. It follows that $s=f \circ (\id,h)$ gives the desired section of $p$.
\end{proof}

\begin{lemma}\label{lemma:direct_summand_pullback}
 Let $A$ be a commutative $R$-algebra, let $I \subseteq A$ be an ideal, and let $\varphi \colon B \rightarrow A \slash I$ be a homomorphism of $R$-algebras. Then every object of the category
\[
 \Mod^{\fp}_B \pb{\Mod^{\fp}_{A \slash I}} \Mod^{\fp}_{A}
\]
 is an epimorphic quotient of a finite direct sum of copies of the unit object.
\end{lemma}

\begin{proof}
 The unit object is given by $(B,\varphi, A)$ where $\varphi \colon A \slash I \ten{B} B \rightarrow A \slash I$ denotes the canonical isomorphism induced by $\varphi \colon B \rightarrow A \slash I$. Thus the goal is to construct an epimorphism
\[
 (B^{\oplus n}, \varphi^{\oplus n}, A^{\oplus n}) \rightarrow (N,f,M)
\]
 for any finitely presentable $B$-module $N$, finitely presentable $A$-module $M$, and isomorphism $f \colon A \slash I \ten{B} N \rightarrow M \slash IM$. Assume that both $M$ and $N$ are generated by $d$ elements. Pick an epimorphism $p\colon B^{\oplus d} \rightarrow N$ of $B$-modules. There exists an epimorphism $\overline{q}$ making the diagram
\[
 \xymatrix{A \slash I \ten{B} B^{\oplus d} \ar[r]^-{\varphi^{\oplus d}} \ar[d]_{A \slash I \ten{B} p} & A^{\oplus d} \slash I A^{\oplus d} \ar[d]^{\overline{q}} \\ A \slash I \ten{B} N \ar[r]_-{f} & M \slash IM}
\]
 commutative since $A \slash IA$ is projective. It is induced by some morphism $q \colon A^{\oplus d} \rightarrow M$. Although the morphism $q$ need not be an epimorphism of $A$-modules, the assumption that $M$ is generated by $d$ elements implies that there exists a morphism $\ell \colon A^{\oplus d} \rightarrow IM$ of $A$-modules such that
\[
 (q,\ell) \colon A^{\oplus 2d} \rightarrow M
\]
 is surjective (indeed, for each generator $m \in M$ we can find an element $a \in A^{\oplus d}$ such that $m-q(a) \in IM$). Together with the the morphism $(p,0) \colon B^{\oplus 2d} \rightarrow N$ we get the desired epimorphism
\[
 (B^{\oplus 2d}, \varphi^{\oplus 2d}, A^{\oplus 2d}) \rightarrow (N,f,M) 
\]
 in $\Mod^{\fp}_B \pb{\Mod^{\fp}_{A \slash I}} \Mod^{\fp}_{A}$.
\end{proof}

 These two lemmas allow us to prove that pushouts of affine schemes along closed immersions are again affine in the 2-category of Adams stacks.

\begin{thm}\label{thm:affine_pushout}
 Let 
\[
 \xymatrix{A \pb{C} B \ar[r] \ar[d] & B \ar[d] \\ A \ar[r] & C}
\]
 be a pullback diagram in the category of commutative $R$-algebras such that $A \rightarrow C$ is surjective and $C \neq 0$. Then the corresponding diagram
\[
 \xymatrix{\Spec(C) \ar[d] \ar[r] & \Spec(A) \ar[d] \\ \Spec(B) \ar[r] & \Spec(A \pb{C} B)}
\]
 is a bicategorical pushout diagram in the 2-category $\ca{AS}$ of Adams stacks. In particular, if $X$ is any Adams stack, then the canonical functor
\[
 X(A \pb{C} B) \rightarrow X(A) \pb{X(C)} X(B)
\]
 is an equivalence of groupoids.
\end{thm}

\begin{proof}
 The assumption that the morphism $A \rightarrow C$ is surjective means that it suffices to consider the case where this morphism is the projection $A \rightarrow A \slash I$ for some ideal $I \subseteq A$. 

 As in the proof of Theorem~\ref{thm:infinite_affine_coproducts} we can use \cite[Theorem~5.6.4]{SCHAEPPI_COLIMITS} to show that there exists an equivalence
\[
 \QCoh_{\fp}\bigr( \Spec(A) \po{\Spec(A \slash I)} \Spec(B) \bigr) \simeq T\bigl( \Mod^{\fp}_B \pb{\Mod^{\fp}_{A \slash I}} \Mod^{\fp}_{A} \bigr)
\]
 of weakly Tannakian categories. The endomorphism ring of the unit of the category $\Mod^{\fp}_B \pb{\Mod^{\fp}_{A \slash I}} \Mod^{\fp}_{A}$ is isomorphic to $A \pb{A \slash I} B$. This reduces the problem to checking that the weakly Tannakian category above is affine, that is, that the conditions of Proposition~\ref{prop:affine_criterion} are satisfied.

 The second condition was proved in Lemma~\ref{lemma:direct_summand_pullback}. To see that the first condition holds, let
\[
 (p,q) \colon (N,f,M) \rightarrow (N^{\prime}, f^{\prime}, M^{\prime})
\]
 be an epimorphism between locally free objects of constant finite rank in the category $\Mod^{\fp}_B \pb{\Mod^{\fp}_{A \slash I}} \Mod^{\fp}_{A}$. Since the functors involved are right exact, this is the case if and only if both $p$ and $q$ are epimorphisms. By Proposition~\ref{prop:locally_free_in_limit}, both $N^{\prime}$ and $M^{\prime}$ are finitely generated projective, so there exists a section $s$ of $p$. By Lemma~\ref{lemma:lifting_sections} we can find a section $t$ of $q$ such that $(s,t)$ is a morphism in the category $\Mod^{\fp}_B \pb{\Mod^{\fp}_{A \slash I}} \Mod^{\fp}_{A}$. This shows that the first condition of Proposition~\ref{prop:affine_criterion} is also satisfied, as claimed. 
\end{proof}

\subsection{Completions}

 Our next example shows that under relatively mild conditions on a sequence of proper ideals
\[
 I_0 \supseteq I_1 \supseteq \ldots \supseteq I_n \supseteq \ldots
\]
 of a commutative $R$-algebra $A$, the canonical morphisms
\[
 \Spec(A \slash I_n) \rightarrow \Spec(\lim A \slash I_n)
\]
 exhibit $\Spec(A)$ as colimit of the chain
\[
 \xymatrix{ \Spec(A \slash I_0) \ar[r] & \Spec(A \slash I_1) \ar[r] & \ldots \ar[r] & \Spec(A \slash I_n) \ar[r] & \ldots }
\]
 in the 2-category $\ca{AS}$ of Adams stacks. This result was proved for all quasi-compact quasi-separated algebraic spaces by Bhatt, and for quasi-compact semi-separated Artin stacks by Hall and Rydh (under the assumption that the commutative $R$-algebra $A$ is noetherian). In both these cases, the authors were able to prove this without demanding that the algebraic space or stack satisfies the resolution property.

 For non-noetherian rings, the category of finitely presentable modules over the completion of a commutative $R$-algebra rather difficult to describe. We will use the fact that the category of finitely generated \emph{projective} modules is much more tractable. Since we are dealing with Adams stacks, it is sufficient to understand the latter (in other words, we can reduce the problem to studying vector bundles, which, in the affine case, are precisely the finitely generated projective modules).

 Since our result holds for arbitrary commutative $R$-algebras $A$, we can apply it to rings of power series $A \llbracket x \rrbracket$ and rings of Witt vectors (cf.\ \cite[Corollary~4.4]{BHATT}).

\begin{lemma}\label{lemma:jacobson}
 Let $A$ be a commutative $R$-algebra and $I \subseteq A$ an ideal which is contained in the Jacobson radical (the intersection of all maximal ideals of $A$). Let $M$ and $N$ be finitely generated projective $A$-modules, and let $\overline{p} \colon M \slash IM \rightarrow N \slash IN$ be an epimorphism of $A \slash I$-modules with section $\overline{s} \colon N \slash IN \rightarrow M \slash IM$.

 Then there exists a lift $p \colon M \rightarrow N$ of $\overline {p}$, and for any such lift there exists a lift $s \colon N \rightarrow M$ of $\overline{s}$ such that $ps=\id$.
\end{lemma}

\begin{proof}
 Since $M$ is projective, we can always find a lift $p$. The cokernel $K$ of $p$ is finitely generated since $N$ is. The assumption that $\overline{p}$ is an epimorphism implies that $K \slash IK=0$, that is, $K=IK$. Since $I$ is contained in the Jacobson radical, it follows from Nakayama's lemma that $K=0$. Thus $p$ is an epimorphism, and the existence of a compatible section $s \colon N \rightarrow M$ of $p$ is a consequence of Lemma~\ref{lemma:lifting_sections}.
\end{proof}

\begin{thm}\label{thm:completion}
 Let $A$ be a commutative $R$-algebra and let
\[
 I_0 \supseteq I_1 \supseteq \ldots \supseteq I_n \supseteq \ldots
\]
 be a sequence of proper ideals of $A$ such that for all maximal ideals $\mathfrak{m}$ of $A$ and all $n \in \mathbb{N}$, the implication
\[
 I_{n+1} \subseteq \mathfrak{m} \Rightarrow I_n \subseteq \mathfrak{m}
\]
 holds. Then the canonical morphisms
\[
 \Spec(A \slash I_n) \rightarrow \Spec(\lim A \slash I_n)
\]
 exhibit $\Spec(\lim A\slash I_n)$ as colimit of the chain
\[
 \xymatrix{ \Spec(A \slash I_0) \ar[r] & \Spec(A \slash I_1) \ar[r] & \ldots \ar[r] & \Spec(A \slash I_n) \ar[r] & \ldots }
\]
 in the 2-category $\ca{AS}$ of Adams stacks. In other words, for each Adams stack $X$, the canonical functors
\[
 X(\lim A \slash I_n) \rightarrow \lim X(A \slash I_n)
\]
 is an equivalence of groupoids.
\end{thm}

\begin{proof}
 We know from \cite[Theorem~5.6.4]{SCHAEPPI_COLIMITS} that the colimit in the 2-category of Adams stacks corresponds to the weakly Tannakian category $T \lim \Mod^{\fp}_{A \slash I_n}$. To prove the claim we need to show that it is affine (see Proposition~\ref{prop:affine_criterion}) and that the endormophism ring of the unit of $\lim \Mod^{\fp}_{A \slash I_n}$ is isomorphic to $\lim A \slash I_n$. The unit is given by the system $(A \slash I_n, \varphi_n)_{n \in \mathbb{N}}$ where $\varphi_n$ denotes the canonical isomorphism
\[
(A \slash I_{n+1}) \slash (I_{n} \slash I_{n+1}) \cong A \slash I_{n} \smash {\rlap{,}}
\]
 so the latter claim follows from the fact that endomorphisms of $A \slash I_n$, considered as a module over itself, are given by multiplication with an element of $A \slash I_n$.

 To see that $T \lim \Mod^{\fp}_{A \slash I_n}$ is affine it remains to check that the two conditions of Proposition~\ref{prop:affine_criterion} are satisfied. From Proposition~\ref{prop:locally_free_in_limit} we know that locally free objects of $\lim \Mod^{\fp}_{A \slash I_n}$ are precisely the systems $(M_n,\varphi_n)$ such that $M_n$ is a finitely generated and projective $A \slash I_n$-module of constant rank. The assumption on maximal ideals means that $I_{n+1} \slash I_n$ is contained in the Jacobson radical of $A \slash I_n$. Thus we can apply Lemma~\ref{lemma:jacobson} inductively, which implies that every epimorphism with target a locally free object of constant finite rank is split.

 Moreover, if we have a system $(M_n,\varphi_n)$ such that $M_n$ is a finitely generated and projective $A \slash I_n$-module of constant rank, then we can find a $d \in \mathbb{N}$ and an epimorphism
\[
p_0 \colon A^{\oplus d} \slash I_0 A^{\oplus d} \rightarrow M_0 
\]
 with section $s_0 \colon M_0 \rightarrow A^{\oplus d} \slash I_0 A^{\oplus d}$. Using Lemma~\ref{lemma:jacobson} we can lift this (inductively) to a system of compatible epimorphisms
\[
p_n \colon A^{\oplus d} \slash I_n A^{\oplus d} \rightarrow M_n 
\]
 with compatible sections $s_n$. By construction, these morphisms $p_n$ and $s_n$ exhibit $(M_n,\varphi_n)$ as direct summand of the $d$-fold direct sum of copies of the unit object of $\lim \Mod^{\fp}_{A \slash I_n}$.
\end{proof}

%\begin{cor}\label{cor:admissible_ring}
% Let $A$ be an admissible $R$-algebra and let $I$ be an ideal of definition of $A$. Then for every Adams stack $X$, the canonical comparison functor
%\[
% X(A) \rightarrow X(A \slash I^n)
%\]
% exhibit the groupoid $X(A)$ as bicategorical limit of the groupoids $X(A \slash I^n)$.
%\end{cor}

\begin{cor}\label{cor:adic_ring}
 Let $A$ be an $R$-algebra which is $I$-adically complete for some ideal $I \subseteq A$ and let $X$ be an Adams stack over $R$. Then the natural functor
\[
 X(A) \rightarrow \lim X(A \slash I^n)
\]
 is an equivalence of groupoids.
\end{cor}

\begin{proof}
  We have to show that for any Adams stack $X$ and any $I$-adic ring $A$, the natural functor
\[
 X(A) \rightarrow \lim X(A \slash I^n)
\]
 is an equivalence of groupoids.

 Since $A$ is $I$-adic, we know that $A \cong \lim A \slash I^n$. We can apply Theorem~\ref{thm:completion} to the sequence of ideals $I_n \defl I^n$ since a maximal ideal contains the $n$-fold product $I^n$ of $I$ if and only if it contains $I$.
\end{proof}

\section{Global version of the Beauville--Laszlo theorem}\label{section:beauville_laszlo}

 In \cite{BEAUVILLE_LASZLO}, the authors claim that \cite[Theor\`eme~3]{BEAUVILLE_LASZLO} can be turned into a global statement as follows. For a scheme $X$ and an effective Cartier divisor $Z \subseteq X$ with sheaf of ideals $I \subseteq \U_X$, we can form the algebra 
\[
\widehat{\U}_X=\lim \U_X \slash I^n
\]
 in the category of quasicoherent sheaves on $X$. The authors claim that for $\widehat Z=\Spec_X (\widehat{\U}_X)$ and $U=X\setminus Z$, the fiber square
\[
 \xymatrix{ U \pb{X} \widehat{Z} \ar[r] \ar[d] & \widehat{Z} \ar[d] \\ U \ar[r] & X }
\]
 satisfies descent for vector bundles. However, this cannot be reduced to the case where $X$ is affine since the formation of $\widehat{Z}$ does \emph{not} commute with restriction to open subsets. It would commute if the limit $\widehat{\U}_X=\lim \U_X \slash I^n$ were computed in the category of $\U_X$-modules instead, but then it would not be possible to form the relative spectrum $\Spec_X(\widehat{\U}_X)$. This problem is also discussed by Ben-Bassat and Temkin in \cite{BASSAT_TEMKIN}.

 In this section we will prove that the conclusion is nevertheless true for a large class of schemes and stacks $X$. The main technical fact we need is that the category $\QCoh(X)$ of quasi-coherent sheaves on $X$ is generated by flat objects. We will also need some mild finiteness conditions, all of which are satisfied if $X$ is an Adams stack.

 Since we cannot work locally, all our constructions and proofs have to take place in the tensor category $\QCoh(X)$ directly. In particular, we have to give meaningful descriptions of all the objects involved in that context. Once we establish some basic properties of these---all of which are well-known in the case where $X$ is affine---we will see that the argument given in \cite{BEAUVILLE_LASZLO} for the affine case can be adapted to the global situation.

 In fact, we can work in the more general context of lfp abelian tensor categories $\ca{C}$ (subject to some conditions). Thus our results apply for example also to certain categories of chain complexes in $\QCoh(X)$ or to categories of Mackey functors. We start by introducing the basic concepts and constructions in a way that makes sense in the context of arbitrary tensor categories.

\subsection{The basic constructions and their properties}\label{section:beauville_laszlo_constructions} 

 Throughout, we fix a lfp abelian tensor category $\ca{C}$. An \emph{effective cartier divisor} in $\ca{C}$ is a subobject $f \colon I \rightarrow \U$ of the unit object which is invertible as an object of $\ca{C}$. Note that since the symmetry acts trivially on $\U$, it must also act trivially on $I$, that is, $I$ is a line bundle in the sense of \cite[Definition~5.3.7]{SCHAEPPI_COLIMITS}. Note that for the case $\ca{C}=\QCoh(X)$, the closed substack $Z$ corresponds to $\Spec_X(\U_X \slash I)$.

 We first need an explicit description of the open complement of $Z$ which generalizes to the context of abstract tensor categories $\ca{C}$. Since the inclusion $X \setminus Z \rightarrow X$ is always affine, it suffices to construct the corresponding commutative algebra in $\ca{C}$. Namely, let $\U_f$ be the pushout
\[
 \xymatrix{\Sym \U \ar[r]^{\Sym f^{\vee}} \ar[d] & \Sym(I^{\vee}) \ar[d] \\ \U \ar[r] & \U_f }
\]
 in the category of commutative algebras in $\ca{C}$ (see also \cite[Definition~5.3.1]{SCHAEPPI_COLIMITS}). Since $I^{\vee}$ is a line bundle, the symmetry acts trivially on $(I^{\vee})^{\otimes n}$. Thus the symmetric algebra on $I^{\vee}$ coincides with the tensor algebra on $I^{\vee}$, and it follows that the underlying object of $\U_f$ is given by the colimit of the chain
\[
 \xymatrix@C=50pt{\U \ar[r]^-{f^{\vee}} & I^{\vee} \ar[r]^-{f^{\vee} \otimes I^{\vee}} & (I^{\vee})^{\otimes 2} \ar[r]^-{f^{\vee} \otimes (I^{\vee})^{\otimes 2}} & \ldots }
\]
 in $\ca{C}$ (see \cite[Lemma~5.14]{SCHAEPPI_GEOMETRIC}). In particular, the object $\U_f$ is flat. We denote its unit by $\eta_f \colon \U \rightarrow \U_f$. We call an object $M \in \ca{C}$ \emph{$f$-regular} (respectively \emph{$f$-local}) if the morphism
\[
 f \otimes M \colon I \otimes M \rightarrow M
\]
 is a monomorphism (isomorphism).

\begin{lemma}\label{lemma:properties_of_o_f}
 Let $\ca{C}$ be a lfp abelian tensor category and let $f \colon I \rightarrow \U$ be an effective Cartier divisor in $\ca{C}$. Then the algebra $\U_f$ has the following properties:
\begin{enumerate}
 \item[(i)] The algebra $\U_f$ is universal among algebras $B$ such that the base change $f^{\vee}_B$ has a section.
\item[(ii)] The base change $f_{\U_f} \colon I_{\U_f} \rightarrow \U_f$ of $f$ is an isomorphism.
\item[(iii)] The morphism $\eta_f \otimes \U_f \colon \U_f \rightarrow \U_f \otimes \U_f$ is an isomorphism.
\item[(iv)] The base change
\[
 (-)_{\U_f} \colon \ca{C} \rightarrow \ca{C}_{\U_f} 
\]
 exhibits $\ca{C}_{\U_f}$ as reflective subcategory of $\ca{C}$. An object $M \in \ca{C}$ lies in the image of the right adjoint $\ca{C}_{\U_f} \rightarrow \ca{C}$ if and only if $M$ is $f$-local.
 \item[(v)] An object $M \in \ca{C}$ is $f$-regular if and only if $\eta_f \otimes M \colon M \rightarrow \U_f \otimes M$ is a monomorphism.
\item[(vi)] A finitely generated object $M \in \ca{C}$ satisfies $\U_f \otimes M \cong 0$ if and only if there exists an $n \in \mathbb{N}$ such that
\[
 M \otimes f^{\otimes n} \colon M \otimes I^{\otimes n} \rightarrow M
\]
 is the zero morphism.
\end{enumerate}
\end{lemma}

\begin{proof}
 The first claim follows from the construction of $\U_f$ and \cite[Proposition~5.3.2]{SCHAEPPI_COLIMITS}.

 To see that $f_{\U_f}$ is an isomorphism, note that by (i), $f^{\vee}_{\U_f}$ has a section, so its dual $f_{\U_f}$ is a split epimorphism. But $f$ is monic and $\U_f$ is flat, so $f_{\U_f}$ is also monic and therefore an isomorphism.

 The unit $\eta_f \colon \U \rightarrow \U_f$ is given by the colimit of the morphisms
\[
(f^{\vee})^{\otimes n}\colon \U \rightarrow (I^{\vee})^{\otimes n} \smash {\rlap{,}} 
\]
 so by (ii) it is sent to an isomorphism by $-\otimes \U_f$. This shows both that (iii) holds and that the multiplication $\mu \colon \U_f \otimes \U_f \rightarrow \U_f$ is an isomorphism.

 To see that (iv) holds, let $M$ be an $\U_f$-module. Then (as an object of $\ca{C}$), $M$ is a retract of $\U_f \otimes M$. Since
\[
 \U_f \otimes M \otimes f \colon  \U_f \otimes M \otimes I \rightarrow \U_f \otimes M
\]
 is an isomorphism, so is $M \otimes f$, that is, $M$ is $f$-local. Arguing as in the proof of (iii) above, this implies that $\eta_f \otimes M$ is an isomorphism. This proves in particular that $M$ underlies an $\U_f$-module if and only if it is $f$-local. Moreover, it shows that the action morphism $\U_f \otimes M \rightarrow M$ of any $\U_f$-module is an isomorphism. Thus the counit of the base change adjunction is invertible, which shows that $\ca{C}_{\U_f}$ is (equivalent to) the reflective subcategory of $f$-local objects.

 To see that (v) holds, note that if $M \otimes f$ is monic, then so is $\eta_f \otimes M$, again by arguing as in the proof of (iii) (using the fact that $\ca{C}$ is lfp abelian, hence in particular Grothendieck abelian). Conversely, if $\eta_f \otimes M$ is monic, then $M$ is a subobject of the $f$-local object $\U_f \otimes M$. But $f$-local objects are in particular $f$-regular and $f$-regular objects are closed under subobjects.

 Finally, to see (vi), let $M$ be a finitely generated object such that $\U_f \otimes M \cong 0$. Since $M$ is finitely generated, there exists a finitely presentable object $M^{\prime}$ together with an epimorphism $p \colon M^{\prime} \rightarrow M$. By definition of $\U_f$, the object $\U_f \otimes M$ the colimit of the (filtered) diagram
\[
 \xymatrix@C=50pt{M \ar[r]^-{f^{\vee} \otimes M} & I^{\vee} \otimes M \ar[r]^-{f^{\vee} \otimes I^{\vee}\otimes M } & (I^{\vee})^{\otimes 2} \otimes M \ar[r] & \ldots }
\]
 in $\ca{C}$. Both
\[
 \xymatrix{M^{\prime} \ar[r]^-{p} & M \ar[r] & \U_f \otimes M} \quad \text{and} \quad \xymatrix{M^{\prime} \ar[r]^-{0} & M \ar[r] & \U_f \otimes M}
\]
 are equal by assumption, so since $M^{\prime}$ is finitely presentable, they must be equal at some finite stage of the colimit. This shows that there exists an $n \in \mathbb{N}$ such that the composite
\[
 \xymatrix@C=40pt{M^{\prime} \ar[r]^-{p} & M \ar[r]^-{(f^{\vee})^{\otimes n} \otimes M} & (I^{\vee})^{\otimes n} \otimes M }
\]
 is zero. Since $p$ is an epimorphism we find that $(f^{\vee})^{\otimes n}\otimes M=0$. The claim that $f^{\otimes n} \otimes M$ is zero therefore follows from the fact that both $(I^{\vee})^{\otimes n}$ and $\U$ are invertible objects in $\ca{C}$.

 Conversely, if $f^{\otimes n} \otimes M =0$, then its base change
\[
 (f_{\U_f})^{\otimes_{\U_f} n} \ten{\U_f} M_{\U_f}
\]
 is the zero morphism as well. But $f_{\U_f}$ is an isomorphism by (ii), so we must have $M_{\U_f}=\U_f \otimes M \cong 0$.
\end{proof}

 The second crucial ingredient in the Beauville--Laszlo theorem is the infinitesimal neighbourhood $\widehat{Z}$ of the effective Cartier divisor $Z=\Spec_X(\U_X \slash I)$. This is again affine over $X$, so we can define the correponding algebra in any tensor category $\ca{C}$: we denote it by
\[
 \widehat{\U} \defl \lim\nolimits_n \U \slash I^n
\]
 where $\U \slash I^n$ is the cokernel of the morphism $f^{\otimes n} \colon I^{\otimes n} \rightarrow \U$.

 Infinite limits in categories of quasi-coherent sheaves are generally quite difficult to describe. For example, the Mittag-Leffler condition rarely implies exactness of the desired sequences, see \cite[Corollary~1.16]{ROOS_REVISITED}. For this reason we need to work a bit harder to establish the basic facts about $\widehat{\U}$ which are needed to prove the global version of the Beauville--Laszlo theorem. By working in a general tensor category, we avoid having to deal with the particulars of the construction of limits in categories of quasi-coherent sheaves.

 We write $p_n \colon \U \rightarrow \U \slash I^n$ for the cokernel of $f^{\otimes n}$ and $\widehat{\eta} \colon \U \rightarrow \widehat{\U}$ for the unit of $\widehat{\U}$ (that is, $\widehat{\eta}=\lim p_n$).

\begin{lemma}\label{lemma:properties_of_o_hat}
 Let $\ca{C}$ be a lfp abelian tensor category and let $f \colon I \rightarrow \U$ be an effective Cartier divisor in $\ca{C}$. Then the algebra $\widehat{\U}=\lim_n \U \slash I^n$ has the following properties:
\begin{enumerate}
 \item[(i)] The morphism $f_{\widehat{\U}} \colon I_{\widehat{\U}} \rightarrow \widehat{\U}$ is a monomorphism.
\item[(ii)] For each $n \in \mathbb{N}$, the morphism
\[
 \widehat{\eta} \otimes \U \slash I^n \colon \U \slash I^n \rightarrow \widehat{\U} \otimes \U \slash I^n
\]
 is an isomorphism.
\item[(iii)] If $f^{\otimes n} \otimes M \colon I^{\otimes n} \otimes M \rightarrow M$ is the zero morphism for some $n \in \mathbb{N}$, then $\widehat{\eta} \otimes M$ is an isomorphism.
\end{enumerate}
\end{lemma}

\begin{proof}
 Since $I$ is invertible, $\widehat{\U} \otimes f$ is isomorphic to $\lim (\U \slash I^n \otimes f)$. From the universal property of cokernels we get a morphism $g_n \colon I \otimes \U \slash I^n \rightarrow \U \slash I^{n+1}$ such that the diagram
\[
 \xymatrix@C=35pt{I^n \otimes I \ar@{=}[d] \ar[r]^-{f^n \otimes I} & I \ar[d]^f \ar[r]^-{ p_n \otimes I} &  \U \slash I^n \otimes I \ar@{-->}[d]^{g_n} \\ I^{n+1} \ar[r]_-{f^{n+1}} & \U \ar[r]_-{p_{n+1}} & \U \slash I^{n+1} }
\]
 is commutative. Moreover, from the snake lemma and the above diagram it follows that $g_n$ is monic. By construction we have $\U \slash I^n \otimes f=p_{n+1,n} \circ g_n$, where $p_{n+1,n} \colon \U \slash I^{n+1} \rightarrow \U \slash I^n$ denotes the canonical morphism used in the definition of $\widehat{\U}$.

 The limit $\lim(g_n)$ is monic since limits preserve monomorphism. Moreover, $\lim(p_{n+1,n})$ is an isomorphism (this is a fact about limits of countable chains valid in any category whatsoever). This proves that (i) holds.

 To see that (ii) holds, first note that the right exact sequence
\begin{equation}\label{eqn:limit_right_ex_seq}
 \xymatrix@C=35pt{ \widehat{\U} \otimes I^n \ar[r]^-{\widehat{\U}\otimes f^n } & \widehat{\U} \ar[r]^-{\widehat{\U} \otimes p_n} & \widehat{\U} \otimes \U \slash I^n \ar[r] & 0 }
\end{equation}
 is exact by (i). On the other hand, since limits commute with kernels, we have a \emph{left} exact sequence
\begin{equation}\label{eqn:limit_left_ex_seq}
 \xymatrix@C=22pt{0 \ar[r] & \lim (\U \slash I^k \otimes I^n) \ar[rr]^-{\lim (\U \slash I^k \otimes f^n)} && \lim \U \slash I^k \ar[rr]^-{\lim (\U \slash I^k \otimes p_n)} && \lim(\U \slash I^k \otimes \U \slash I^n) }
\end{equation}
 for each fixed $n \in \mathbb{N}$. Moreover, $\lim_k (\U \slash I^k \otimes f^n)$ is isomorphic to $\widehat{\U} \otimes f^n$ since $I^n$ is invertible: the functor $-\otimes I^n$ preserves \emph{all} limits.

 We claim that there exists an isomorphism $\lim_k (\U \slash I^k \otimes \U \slash I^n) \cong \U \slash I^n$ such that the diagram
\[
 \xymatrix@C=40pt{\widehat{\U} \ar[r]^-{\lim(1 \otimes p_n)} & \lim(\U \slash I^k \otimes \U \slash I^n) \ar[d]^{\cong} \\ \U \ar[u]^{\widehat{\eta}} \ar[r]_-{p_n} & \U \slash I^n }
\]
 is commutative. To see this, it suffices to check that the system
\[
 \xymatrix@C=35pt{\U \ar[r]^-{p_k} & \U \slash I^k \ar[r]^-{\U \slash I^k \otimes p_n} & \U \slash I^k \otimes \U \slash I^n}
\]
 is isomorphic to $p_n \colon \U \rightarrow \U \slash I^n$ for $k$ sufficiently large. But as soon as $k \geq n$, we have $f^k \otimes \U \slash I^n=0$, hence $p_k \otimes \U \slash I^n$ is an isomorphism. Thus the limit of 
\[
p_k \otimes \U \slash I^n \colon \U \slash I^n \rightarrow \U \slash I^k \otimes \U \slash I^n 
\]
 induces (the inverse of) the desired isomorphism.

 This shows in particular that the left exact sequence \eqref{eqn:limit_left_ex_seq} is exact. Combining the commutative square above with the exact sequences \eqref{eqn:limit_right_ex_seq} and \eqref{eqn:limit_left_ex_seq} we find that there exists an isomorphism $\widehat{\U} \otimes \U \slash I^n \cong \U \slash I^n$ such that the diagram
\[
 \xymatrix{\widehat{\U} \ar[r]^-{\widehat{\U} \otimes p_n} & \widehat{\U} \otimes \U \slash I^n \ar[d]^{\cong} \\ \U \ar[u]^{\widehat{\eta}} \ar[r]_-{p_n} & \U \slash I^n }
\]
 is commutative. Since $p_n$ is an epimorphism, the inverse of the isomorphism $\widehat{\U} \otimes \U \slash I^n \rightarrow \U \slash I^n$ must be equal to $\widehat{\eta} \otimes \U \slash I^n$.

 Finally, to see (iii), note that $\widehat{\eta} \otimes \U \slash I^n \otimes M$ is an isomorphism by (ii). The assumption on $M$ implies that $p_n \otimes M \colon M \rightarrow \U \slash I^n \otimes M$ is an ismorphism as well. Thus $\widehat{\eta} \otimes M$ must be an isomorphism since the diagram
\[
 \xymatrix@C=35pt{M \ar[d]_{\widehat{\eta} \otimes M} \ar[r]^-{p_n \otimes M}_-{\cong} & \U \slash I^n \otimes M \ar[d]^{\widehat{\eta} \otimes \U \slash I^n \otimes M }_{\cong}  \\
 \widehat{\U} \otimes M \ar[r]_-{\widehat{\U} \otimes p_n \otimes M}^-{\cong} & \widehat{\U} \otimes \U \slash I^n \otimes M}
\]
 is commutative.
\end{proof}

 From now on, the proof of the global Beauville--Laszlo theorem follows the proof strategy of the affine case in \cite{BEAUVILLE_LASZLO}, replacing the algebras $A$, $A\slash(f^n)$, $A_f$, and $\widehat{A}$ by $\U$, $\U \slash I^n$, $\U_f$, and $\widehat{\U}$ in $\ca{C}$. Since this argument is quite involved and requires some modifications, we provide full details below. The following gives a global version of \cite[Lemme~1]{BEAUVILLE_LASZLO}.

\begin{lemma}\label{lemma:faithfulness_of_o_f_and_o_hat}
 Let $\ca{C}$ be a lfp abelian tensor category and let $f \colon I \rightarrow \U$ be an effective Cartier divisor in $\ca{C}$. For any object $M \in \ca{C}$ with $\U_f \otimes M \cong 0$, the morphism $\widehat{\eta} \otimes M \colon M \rightarrow \widehat{\U} \otimes M$ is an isomorphism. In particular, the algebras $\U_f$ and $\widehat{\U}$ jointly detect non-zero objects: if $\U_f\otimes M \cong 0$ and $\widehat{\U} \otimes M \cong 0$, then $M \cong 0$. 
\end{lemma}

\begin{proof}
 Suppose that $\U_f \otimes M \cong 0$ and write $M$ as directed union of finitely generated subobjects $M_i$. Since $\U_f$ is flat, we have $\U_f \otimes M_i \cong 0$ as well. Thus for each $i$ there exists a $n_i \in \mathbb{N}$ such that $f^{\otimes n_i} \otimes M_i$ is the zero morphism (see Part~(vi) of Lemma~\ref{lemma:properties_of_o_f}). From Part~(iii) of Lemma~\ref{lemma:properties_of_o_hat} it follows that $\widehat{\eta} \otimes M_i$ is an isomorphism for all $i$. Since tensor products commute with colimits, this implies that $\widehat{\eta} \otimes M$ is an isomorphism. In particular, if we also have $\widehat{\U} \otimes M \cong 0$, then $M \cong 0$, as claimed.  
\end{proof}

 Before we can proceed to prove the global versions of the remaining lemmas in \cite{BEAUVILLE_LASZLO} we need to establish some basic facts about flat objects in abelian categories.

\subsection{Abelian categories with enough flat objects}
 It is well-known that in any symmetric monoidal abelian categories with enough flat objects the theory of Tor-functors can be developed. The author was however unable to find a precise reference for the exact statement needed in our context. While the results follow more or less directly from the machinery developed in \cite{SGA_XVII}, it is also possible to give a more elementary account in the case we care about. We therefore sketch a geodesic proof of the basic statements in the theory of Tor-functors which closely follows the development of the classical case for module categories. The main difference is that we do not have enough projectives, so we have to work a bit harder to show that our functors do not depend on the chosen flat resolution. Recall that an object $M$ in a symmetric (or braided) monoidal abelian category is called flat if $M \otimes -$ is exact. We say that $\ca{C}$ \emph{has enough flat objects} if for all $M \in \ca{C}$ there exists a flat object $F$ and an epimorphism $F \rightarrow M$.

\begin{lemma}\label{lemma:exact_sequences_with_flat_objects}
 Let $\ca{C}$ be a symmetric (or braided) monoidal abelian category such that the functor $M \otimes -$ is right exact for each $M \in \ca{C}$. Then for any $C \in \ca{C}$ and any exact sequence
\begin{equation}\label{eqn:exact_seq_with_flat}
 \xymatrix{0 \ar[r] & L \ar[r] & M \ar[r] & N \ar[r] & 0}  
\end{equation}
 where $M$ and $N$ are flat, the sequence
\begin{equation}\label{eqn:tensored_exact_seq_with_flat}
 \xymatrix{0 \ar[r] & C\otimes L \ar[r] & C\otimes M \ar[r] & C\otimes N \ar[r] & 0}
\end{equation}
 is exact. Moreover, the object $L$ is flat as well.
\end{lemma}

\begin{proof}
 To see that the sequence \eqref{eqn:tensored_exact_seq_with_flat} is exact, first write $C$ as cokernel of a morphism $C_0 \rightarrow C_1$ with $C_0$ and $C_1$ flat. Applying the snake lemma to the diagram
\[
 \xymatrix{0 \ar[r] & C_0 \otimes L \ar[r] \ar[d] & C_0 \otimes M \ar[r] \ar[d] & C_0 \otimes N \ar[r] \ar[d] & 0 \\
0 \ar[r] & C_1\otimes L \ar[r] & C_1\otimes M \ar[r] & C_1\otimes N \ar[r] & 0}
\]
 and using the fact that $M$ and $N$ are flat we find that the sequence \eqref{eqn:tensored_exact_seq_with_flat} is indeed exact.

 To see that $L$ is flat, let $i \colon C \rightarrow C^{\prime}$ be a monomorphism. The horizontal morphisms in the diagram
\[
 \xymatrix{C \otimes L \ar[d]_{i \otimes L} \ar[r] & C \otimes M \ar[d]^{i \otimes M} \\ C^{\prime} \otimes L \ar[r] & C^{\prime} \otimes M}
\]
 are monic by exactness of \eqref{eqn:tensored_exact_seq_with_flat}. The claim that $i \otimes L$ is monic therefore follows from the fact that $M$ is flat.
\end{proof}

 Using this we can check that the evident definition of Tor-functors---using flat resolutions rather than projective ones---is well-defined. Recall that a \emph{flat resolution} $M_\bullet \rightarrow M$ of an object $M \in \ca{C}$ is an exact sequence
\[
 \xymatrix{\ldots \ar[r] & M_i \ar[r] & \ldots \ar[r] & M_0 \ar[r] & M \ar[r] &0}
\]
 such that each $M_i$ is flat for all $i \in \mathbb{N}$.

\begin{lemma}\label{lemma:independece_of_flat_resolution}
 Let $\ca{C}$ be an abelian category with symmetric (or braided) monoidal structure such that the functor $M \otimes -$ is right exact for each $M\in \ca{C}$. Suppose that $\ca{C}$ has enough flat objects. Let $M_\bullet$ and $M^{\prime}_\bullet$ be flat resolutions of $M$. Then there exists a flat resolution $F_\bullet$ of $M$, together with morphisms
\[
 \xymatrix{M_{\bullet} & F_{\bullet} \ar[l]_-{p_{\bullet}} \ar[r]^-{q_{\bullet}} & M^{\prime}_{\bullet} }
\]
 which are both quasi-isomorphisms and levelwise epimorphisms. Moreover, for any $N \in \ca{C}$, the morphisms $p_{\bullet} \otimes N$ and $q_{\bullet} \otimes N$ are quasi-isomorphisms.
\end{lemma}

\begin{proof}
 Using the fact that $\ca{C}$ has enough flat objects, we can find a flat object $F_0$ as in the lower right square of the diagram
\[
 \xymatrix{P_0 \ar@{=}[r] \ar@{=}[d] & P_0 \ar[d] \ar@{->>}[r] & K_0^{\prime} \ar[d]\\ 
 P_0 \ar[r] \ar@{->>}[d] & F_0 \ar@{->>}[d] \ar@{->>}[r] & M_0^{\prime} \ar@{->>}[d] \\
 K_0 \ar[r] & M_0 \ar@{->>}[r] & M
}
\]
 in $\ca{C}$. We define $P_0$ as pullback of the kernel $K_0$ (equivalently, of the kernel $K_0^{\prime}$). By omitting $M_0$ (respectively $M_0^{\prime}$), we obtain flat resolutions of $K_0$ (respectively of $K_0^{\prime}$). We can pull these back to obtain two flat resolutions of $P_0$. More precisely, the pullback $Q \rightarrow P_0$ of the epimorphism $M_1 \rightarrow K_0$ has kernel (isomorphic to) the kernel $K_1$ of $M_1 \rightarrow M_0$, so we can use the $M_i$, $i \geq 2$ to complete $Q \rightarrow P_0$ to a resolution of $P_0$. The object $Q$ fits in the pullback diagram
\[
\xymatrix{Q \ar[r] \ar@{->>}[d] & F_0 \ar@{->>}[d] \\ M_1 \ar[r] & M_0 } 
\]
 and is therefore flat by Lemma~\ref{lemma:exact_sequences_with_flat_objects}, so we indeed obtain a flat resolution of $P_0$. The same construction can be applied to the resolution $M_{\bullet}^{\prime}$.

 By iterating this procedure, we obtain a flat resolution $F_{\bullet}$ of $M$, together with levelwise epimorphisms $p_{\bullet}$ and $q_{\bullet}$ to $M_{\bullet}$ and $M^{\prime}_{\bullet}$. These are easily seen to be quasi-isomorphisms: by construction we have $H_i(F_{\bullet})\cong 0$ for $i > 0$, so we only need to check that $H_0(F_{\bullet})$, the cokernel of $F_1 \rightarrow F_0$, is isomorphic to $M$. This follows from the fact that $P_0$ is the kernel of the epimorphism $F_0 \rightarrow M$.

 It remains to check that for any $N \in \ca{C}$, the two morphisms $p_{\bullet} \otimes N$ and $q_{\bullet} \otimes N$ are quasi-isomorphisms. Let $L_{\bullet}$ be the kernel of $p_{\bullet}$. By Lemma~\ref{lemma:exact_sequences_with_flat_objects}, $L_i$ is flat for all $i \in \mathbb{N}$. Moreover, the long exact homology sequence induced by the exact sequence
\[
 \xymatrix{0 \ar[r] & L_{\bullet} \ar[r] & F_{\bullet} \ar[r] & M_{\bullet} \ar[r] & 0 }
\]
 shows that $L_{\bullet}$ is exact. This implies that the kernel of each morphism $L_{i+1} \rightarrow L_i$ is flat (by using induction on $i$ and the second statement of Lemma~\ref{lemma:exact_sequences_with_flat_objects}). From the first statement of Lemma~\ref{lemma:exact_sequences_with_flat_objects} it therefore follows that $L_{\bullet} \otimes N$ is exact.

 The long exact homology sequence induced by the exact sequence
\[
 \xymatrix{0 \ar[r] & L_{\bullet} \otimes N \ar[r] & F_{\bullet} \otimes N \ar[r]^-{p_{\bullet} \otimes N} & M_{\bullet} \otimes N \ar[r] & 0 }
\]
 of complexes therefore shows that $p_{\bullet} \otimes N$ is a quasi-isomorphism. The same reasoning applies to $q_{\bullet} \otimes N$.
\end{proof}

 Note that a Grothendieck abelian category with enough flat objects in particular has a generating set consisting of flat objects. Using such a generating set and the fact that filtered colimits are exact, we can find \emph{functorial} flat resolutions
\[
F(M)_{\bullet} \rightarrow M 
\]
 of $M \in \ca{C}$ by iteratively taking the (infinite) direct sum over all morphisms whose domain is one of the flat generators and whose target is $M$ (respectively the kernel of $F(M)_{i+1} \rightarrow F(M)_{i}$).

\begin{dfn}\label{dfn:tor_functors}
 Let $\ca{C}$ be a symmetric monoidal closed Grothendieck abelian category with enough flat objects. Then the \emph{$n$-th Tor-functor}
\[
 \Tor_n^{\ca{C}}(-,-) \colon \ca{C}\times \ca{C} \rightarrow \ca{C}
\]
 is defined by $Tor_n^{\ca{C}}(M,N)=H_n\bigl( F(M)_{\bullet}\otimes N \bigr)$.
\end{dfn}

 Using the above two lemmas we can easily check that these Tor-functors have the expected properties.

\begin{prop}\label{prop:tor_properties}
 Let $\ca{C}$ be a symmetric monoidal closed Grothendieck abelian category with enough flat objects. Let $M$ and $N$ be objects in $\ca{C}$ with flat resolutions $M_{\bullet} \rightarrow M$ and $N_{\bullet} \rightarrow N$.
\begin{enumerate}
 \item[(i)] There exists a natural isomorphism $\Tor_0^{\ca{C}}(M,N) \cong M \otimes N$.
\item[(ii)] For each short exact sequence
\[
 \xymatrix{0 \ar[r] & A \ar[r] & B \ar[r] & C \ar[r] & 0 }
\]
 in $\ca{C}$ there is a long exact sequence
\[
 \xymatrix@C=15pt{\ldots \ar[r] & \Tor_n^{\ca{C}}(M,A) \ar[r] & \Tor_n^{\ca{C}}(M,B) \ar[r] & \Tor_n^{\ca{C}}(M,C) \ar[r] & \Tor_{n-1}^{\ca{C}}(M,A) \ar[r] & \ldots }
\]
 in $\ca{C}$.
\item[(iii)] There are isomorphisms $\Tor_n^{\ca{C}}(M,N) \cong H_n(M_{\bullet} \otimes N)\cong H_n(M \otimes N_{\bullet})$. In particular, $\Tor_n^{\ca{C}}(M,N) \cong \Tor_n^{\ca{C}}(N,M)$, that is, the Tor-functors are symmetric.
\item[(iv)] An object $M \in \ca{C}$ is flat if and only if $\Tor_1^{\ca{C}}(M,C) \cong 0$ for all objects $C \in \ca{C}$.
\item[(v)] If in the short exact sequence
\[
 \xymatrix{0 \ar[r] & A \ar[r] & B \ar[r] & C \ar[r] & 0 }
\]
 the object $C$ is flat and $L \in \ca{C}$ is an arbitrary object, then the sequence
\[
 \xymatrix{0 \ar[r] & L\otimes A \ar[r] & L\otimes B \ar[r] & L\otimes C \ar[r] & 0 } 
\]
 is exact.
\item[(vi)] If $M \cong \colim M_i$ is a filtered colimit of objects $M_i$ and $N$ is an object such that $\Tor_n^{\ca{C}}(M_i, N) \cong 0$, then $\Tor_n^{\ca{C}}(M,N) \cong 0$.
\end{enumerate}
\end{prop}

\begin{proof}
 The isomorphism $\Tor_0^{\ca{C}}(M,N) \cong M\otimes N$ follows from the fact that $-\otimes N$ is right exact. 

 The long exact sequence of (ii) is simply the long exact homology sequence of the exact sequence
\[
 \xymatrix{0 \ar[r] & F(M)_{\bullet} \otimes A \ar[r] & F(M)_{\bullet} \otimes B \ar[r] & F(M)_{\bullet} \otimes B \ar[r] & 0 }
\]
 of complexes in $\ca{C}$.

 The isomorphism $\Tor_n^{\ca{C}}(M,N) \cong H_n(M_\bullet \otimes N)$ follows directly from Lemma~\ref{lemma:independece_of_flat_resolution}, while the isomorphism $H_n(M_{\bullet} \otimes N) \cong H_n(M \otimes N_{\bullet})$ follows as usual from the spectral sequence associated to the double complex $M_{\bullet} \otimes N_{\bullet}$. This proves (iii).

 To see that (iv) holds, note that (iii) implies that $\Tor_1^{\ca{C}}(M,C)\cong 0$ if $M$ is flat. Conversely, if $\Tor_1^{\ca{C}}(M,C)\cong 0$ for all $C \in \ca{C}$, then the long exact Tor-sequence (ii) implies that $M$ is flat. 

 The symmetry (iii) of $Tor_n^{\ca{C}}$ implies that $\Tor_1^{\ca{C}}(L,C) \cong 0$ if $C$ is flat. The claim (v) therefore follows from the long exact Tor-sequence (ii).

 Finally, the fact (vi) about filtered colimits follows from (iii) and the fact that filtered colimits in $\ca{C}$ are exact.
\end{proof}

 Having established these basic facts we are now ready to prove the global versions of \cite[Lemme~2 and 3]{BEAUVILLE_LASZLO}.

\subsection{Global verions of the remaining lemmas}
 The first lemma concerns base change functors $(-)_B \colon \ca{C} \rightarrow \ca{C}_B$ for algebras $B$ which detect non-zero objects.

\begin{lemma}\label{lemma:almost_faithful_algebra}
 Let $\ca{C}$ be a lfp abelian tensor category with enough flat objects and let $B \in \ca{C}$ be a commutative algebra which detects non-zero objects: for all $M \in \ca{C}$, the implication
\[
 M_B \cong 0 \implies M \cong 0
\]
 holds. Then the following hold:
\begin{enumerate}
 \item[(i)] If $p_B \colon M_B \rightarrow N_B$ is an epimorphism, then so is $p$.
\item[(ii)] If $M_B \in \ca{C}_B$ is finitely generated, then so is $M \in \ca{C}$.
\item[(iii)] If $M \in \ca{C}$ is flat and $M_B \in \ca{C}_B$ is finitely presentable, then $M$ is finitely presentable.  
\end{enumerate}
\end{lemma}

\begin{proof}
 Part (i) follows since the cokernel $K$ of $p$ satisfies $K_B \cong 0$.

 To see that (ii) holds, write $M$ as directed union of finitely generated subobjects $M_i$. Then $M_B \cong \colim (M_i)_B$. Note that the morphisms $(M_i)_B \rightarrow M_B$ need not be monic since we did not assume that $B$ is flat. However, $M_B$ is the union of the \emph{images} of these morphisms. Since $M_B$ is finitely generated, it must be equal to one of these images, that is, $(M_i)_B \rightarrow M_B$ is an epimorphism for some index $i$. From (i) it follows that $M_i \rightarrow M$ is an epimorphism, hence an isomorphism. Thus $M\cong M_i$ is finitely generated.

 It remains to check that if $M$ is flat and $M_B$ is finitely presentable, then $M$ is finitely presentable. As we just observed above, the assumption implies that $M$ is finitely generated, so there exists an epimorphism $p \colon M^{\prime } \rightarrow M$ for some finitely presentable object $M^{\prime}$. Let $K$ be the kernel of $p$. To show the claim it suffices to check that $K$ is finitely generated. Since $M$ is flat, the sequence
\[
 \xymatrix{0 \ar[r] & K_B \ar[r] & M^{\prime}_B \ar[r] & M_B \ar[r] & 0}
\]
 is exact by Part~(v) of Proposition~\ref{prop:tor_properties}. Thus $K_B$ is finitely generated (see for example \cite[Lemma~2.1]{SCHAEPPI_INDABELIAN}), and from (ii) it follows that $K$ is finitely generated, as claimed.
\end{proof}

 The final lemma proved in \cite{BEAUVILLE_LASZLO} concerns the cokernel
\[
 \xymatrix{\U \ar[r]^-{\eta_f} & \U_f \ar[r]^-{\pi} & \U_f \slash \U}
\]
 of $\eta_f$. By construction of $\U_f$ (see \S \ref{section:beauville_laszlo_constructions}), this cokernel is isomorphic to the filtered colimit of the diagram
\[
 \xymatrix{0 \ar[r] & K_1 \ar[r] & K_2 \ar[r] & \ldots }
\]
 where $K_i$ denotes the cokernel of $(f^{\vee})^{\otimes n} \colon \U \rightarrow (I^{\vee})^{\otimes n}$.

\begin{lemma}\label{lemma:k_n_properties}
 Let $\ca{C}$ be a lfp abelian tensor category and let $f \colon I \rightarrow \U $ be an effective Cartier divisor in $\ca{C}$. Then the objects $K_n$ defined above have the following properties:
\begin{enumerate}
 \item[(i)] There are canonical isomorphisms $I^{\otimes n} \otimes K_n \cong \U \slash I^n$.
\item[(ii)] For each $n \in \mathbb{N}$, the morphisms
\[
 p_n \otimes K_n \colon K_n \rightarrow \U \slash I^n \otimes K_n \quad\text{and}\quad \widehat{\eta} \otimes K_n \colon K_n \rightarrow \widehat{\U} \otimes K_n
\]
 are isomorphisms. Thus the filtered colimit $\widehat{\eta} \otimes (\U_f \slash \U)$ of the morphisms $\widehat{\eta} \otimes K_n$ is an isomorphism as well.
\end{enumerate}
\end{lemma}

\begin{proof}
 From the commutative square in the diagram
\[
\xymatrix@C=40pt{ I^n \ar[r]^-{I^n \otimes (f^{\vee})^{\otimes n}} \ar@{=}[d] & I^n \otimes (I^{\vee})^n  \ar[d]^{\ev}_{\cong} \ar[r] & I^n \otimes K_n \ar@{-->}[d]^{\cong} \\ 
I^n \ar[r]_-{f^{\otimes n}} & \U \ar[r]_-{p_n} & \U \slash I^n } 
\]
 of short exact sequences we deduce that $I^n \otimes K_n \cong \U \slash I^n$. The second claim therefore follows by tensoring with the invertible object $I^n$ and by using the fact that $\U \slash I^n$ has the corresponding properties: the first isomorphism is immediate from the definition of $p_n$, and the second follows from Part~(ii) of Lemma~\ref{lemma:properties_of_o_hat}.
\end{proof}

 The following is a global version of \cite[Lemme~3]{BEAUVILLE_LASZLO}.

\begin{lemma}\label{lemma:tor_computations}
 Let $\ca{C}$ be a lfp abelian tensor category with enough flat objects and let $f \colon I \rightarrow \U$ be an effective Cartier divisor in $\ca{C}$. Let $\U_f$ and $\widehat{\U}$ be the localization at $f$ and the completion of $\U$ at $I$ respectively. Let $\U_f \slash \U$ denote the cokernel of the unit $\eta_f$ of $\U_f$.
\begin{enumerate}
 \item[(i)] For every $M \in \ca{C}$ we have $\Tor_1^{\ca{C}}\bigl(\widehat{\U},(\U_f \slash \U) \otimes M\bigr) \cong 0$.
\item[(ii)] If $M$ is a flat $\U \slash I^n$-module, then
\[
 \Tor_i^{\ca{C}}(C,M) \cong 0
\]
 for all $C \in \ca{C}$ and all $i \geq 2$.
\item[(iii)] If $G$ is a flat $\widehat{\U}$-module, then
\[
 \Tor_i^{\ca{C}}\bigl(C,(\U_f \slash \U) \otimes G \bigr) \cong 0
\]
 for all $C \in \ca{C}$ and all $i \geq 2$.
\end{enumerate}
\end{lemma}

\begin{proof}
 To see (i) it suffices to check that $\Tor^{\ca{C}}_1(\widehat{\U}, K_n \otimes M) \cong 0$ for all $n$. Moreover, since the invertible object $I^n$ is in particular flat, we can reduce this to checking that
\[
 \Tor_1^{\ca{C}}(\widehat{\U},K_n \otimes M) \otimes I^n \cong  \Tor_1^{\ca{C}}(\widehat{\U},I^n \otimes K_n \otimes M) \cong  \Tor_1^{\ca{C}}(\widehat{\U}, \U \slash I^n \otimes M)
\]
 is zero (see Part~(i) of Lemma~\ref{lemma:k_n_properties}).

 We can choose a flat object $F$ with an epimorphism $p \colon F\rightarrow M$. Let $N$ be the kernel of $\U \slash I^n \otimes p \colon \U \slash I^n \otimes F \rightarrow \U \slash I^n \otimes M$. Then $N \otimes f^{\otimes n}=0$, so the vertical morphisms in the commutative diagram
\[
 \xymatrix@C=35pt{0\ar[r] & N \ar[r] \ar[d]^{\widehat{\eta}\otimes N} & \U \slash I^n \otimes F \ar[d]^{\widehat{\eta} \otimes \U \slash I^n \otimes F} \ar[r]^-{\U \slash I^n \otimes p} & \U \slash I^n \otimes M \ar[d]^{\widehat{\eta} \otimes \U \slash I^n \otimes M} \ar[r] & 0\\ 
& \widehat{\U} \otimes N \ar[r] & \widehat{\U} \otimes \U \slash I^n \otimes F \ar[r]_-{\widehat{\U} \otimes \U \slash I^n \otimes p} & \widehat{\U} \otimes \U \slash I^n \otimes M}
\]
 are all isomorphisms (see Part~(iii) of Lemma~\ref{lemma:properties_of_o_hat}). From the long exact Tor-sequence it follows that there is an epimorphism
\[
  \Tor_1^{\ca{C}}(\widehat{\U},\U \slash I^n \otimes F) \rightarrow  \Tor_1^{\ca{C}}(\widehat{\U}, \U \slash I^n \otimes M) \smash{\rlap{,}}
\]
 which reduces the problem to the case where $M=F$ is flat. In this case we have an isomorphism
\[
\Tor_1^{\ca{C}}(\widehat{\U}, \U \slash I^n \otimes M) \cong \Tor_1^{\ca{C}}(\widehat{\U}, \U \slash I^n ) \otimes M \smash {\rlap{,}}
\]
 so it suffices to check that $\Tor_1^{\ca{C}}(\widehat{\U}, \U \slash I^n) \cong 0$. If we compute this using the flat resolution
\[
 \xymatrix{0 \ar[r] & I^n \ar[r]^-{f^{\otimes n}}& \U \ar[r] & \U \slash I^n} 
\]
 of $\U \slash I^n$, we find that $\Tor_1^{\ca{C}}(\widehat{\U}, \U \slash I^n)$ is the kernel of $\widehat{\U} \otimes f^{\otimes n}$. But $\widehat{\U}\otimes f$ is monic by Part~(i) of Lemma~\ref{lemma:properties_of_o_hat}. This concludes the proof that
\[
\Tor_1^{\ca{C}}\bigl(\widehat{\U}, (\U_f \slash \U )\otimes M\bigr) \cong 0 
\]
 for all $M \in \ca{C}$.

 To see (ii), let $M$ be a flat $\U \slash I^n$-module and let $C \in \ca{C}$ be an arbitray object with flat resolution $C_{\bullet}$. From the isomorphisms
\begin{align*}
 \Tor_i^{\ca{C}}(M,C) & \cong H_i(M \otimes C_{\bullet}) \\
 & \cong H_i \bigl( M \ten{\U \slash I^n} ( \U \slash I_n \otimes C_{\bullet})\bigr) \\
& \cong M \ten{\U \slash I^n} H_i(\U \slash I^n \otimes C_{\bullet}) \\
&\cong M \ten{\U \slash I^n} \Tor_i^{\ca{C}}(\U \slash I^n,C)
\end{align*}
 it follows that it suffices to check the claim for $M=\U \slash I^n$. In this case, the Tor-functor in question can be computed using the 2-term flat resolution
\[
 \xymatrix{0 \ar[r] & I^n \ar[r]^-{f^{\otimes n}}& \U \ar[r] & \U \slash I^n}  
\]
 of $\U \slash I^n$. This shows that $Tor_i^{\ca{C}}(\U \slash I^n ,C)$ is indeed zero for $i \geq 2$.

 It remains to check (iii): that for each flat $\widehat{\U}$-module $G$ we have
\[
 \Tor_i^{\ca{C}}\bigl(C,(\U_f \slash \U)\otimes G\bigr) \cong 0
\]
 for all $C \in \ca{C}$ and $i \geq 2$. By Part~(vi) of Proposition~\ref{prop:tor_properties} it suffices to check this for the objects $K_n$ instead of $\U_f \slash \U$, and by tensoring with the invertible object $I^n$ we find that this is equivalent to the claim that
\[
 \Tor_i^{\ca{C}}(C,\U \slash I^n \otimes G) \cong 0
\]
 for all $n \in \mathbb{N}$, $C \in \ca{C}$, and $i \geq 2$ (see Lemma~\ref{lemma:k_n_properties}). By (ii) it suffices to check that the object
\[
\U \slash I^n \otimes G \cong (\widehat{\U} \otimes \U \slash I^n) \ten{\widehat{\U}} G  
\]
 of $\ca{C}$ has the structure of a flat $\U \slash I^n$-module. Recall from Part~(ii) of Lemma~\ref{lemma:properties_of_o_hat} that the unit of $\widehat{\U}$ gives an isomorphism $\U \slash I^n \cong \widehat{\U} \otimes \U \slash I^n$ of algebras. This reduces the problem to checking that the image of $G$ under the tensor functor
\[
 (-)_{(\U \slash I^n)_{\widehat{\U}}} \colon \ca{C}_{\widehat{\U}} \rightarrow \ca{C}_{(\U \slash I^n)_{\widehat{\U}}}
\]
 is flat. This follows from the fact that for \emph{any} $\widehat{\U}$-algebra $B$, the base change functor $(-)_B \colon \ca{C}_{\widehat{\U}} \rightarrow \ca{C}_B$ preserves flat objects. Indeed, the natural isomorphism
\[
 U(G_B \ten{B} -) \cong G \ten{\widehat{\U}} U(-)
\]
 (where $U$ denotes the right adjoint of the base change functor) shows that $G_B$ is flat as a $B$-module if $G$ is flat as $\widehat{\U}$-module.
\end{proof}

\subsection{The Beauville--Laszlo theorem for tensor categories}
 
 Let $\ca{C}$ be a lfp abelian tensor category with enough flat objects. Let $f \colon I \rightarrow \U$ be an effective Cartier divisor in $\ca{C}$ (that is, $f$ is a monomorphism whose domain is an invertible object). As we have seen above, we can then form the localization $\U_f$ of $\U$ at $f$ and the completion $\widehat{\U} \defl \lim_n \U \slash I^n$ of $\U$ at $I$. We get a tensor functor
\[
 \ca{C}_{\widehat{\U}} \rightarrow  \ca{C}_{\U_f \otimes \widehat{\U}}
\]
 given by tensoring with $\U_f$ and a tensor functor
\[
 \ca{C}_{\U_f} \rightarrow  \ca{C}_{\widehat{\U} \otimes \U_f} \cong \ca{C}_{\U_f \otimes \widehat{\U}} 
\]
 given by tensoring with $\widehat{\U}$. The diagram
\begin{equation}\label{eqn:canonical_iso_in_base_change}
 \vcenter{\xymatrix{\ca{C} \xtwocell[1,1]{}\omit{^\sigma} \ar[r] \ar[d] & \ca{C}_{\widehat{\U}} \ar[d] \\
\ca{C}_{\U_f} \ar[r] & \ca{C}_{\U_f \otimes \widehat{\U}} }}
\end{equation}
 commutes up to natural isomorphism given by the symmetry
\[
 \sigma \otimes M \colon \widehat{\U} \otimes \U_f \otimes M \cong \U_f \otimes \widehat{\U} \otimes M
\]
 in the tensor category $\ca{C}$. Recall that an object $M \in \ca{C}$ is called \emph{$f$-regular} if $f \otimes M$ is a monomorphism.

 In order to simplify the notation we will omit the tensor symbol of $\ca{C}$ and simply write the tensor product of two objects $M$ and $N$ as the concatenation $MN$ throughout this section.

\begin{lemma}\label{lemma:beauville_laszlo_fully_faithful}
 Let $\ca{C}^{f-\reg}$ denote the full subcategory of $\ca{C}$ consisting of $f$-regular objects. The functor
\[
 \ca{C}^{f-\reg} \rightarrow \ca{C}_{\U_f} \pb{\ca{C}_{\U_f \widehat{\U}}} \ca{C}_{\widehat{\U}}
\]
 induced by the natural isomorphism \eqref{eqn:canonical_iso_in_base_change} is fully faithful. Moreover, $\widehat{\U} M$ is $f$-regular if $M$ is $f$-regular.
\end{lemma}

\begin{proof}
 Let $M$ and $N$ be $f$-regular and let $\varphi \colon \U_f M \rightarrow \U_f N$ and $\psi \colon \widehat{\U} M \rightarrow \widehat{\U} N$ constitute a morphism in the target, that is, assume that the diagram
\[
 \xymatrix{\widehat{\U} \U_f M \ar[r] \ar[d]_{\widehat{\U} \varphi} \ar[r]^-{\sigma M} & \U_f \widehat{\U} M \ar[d]^{\U_f \psi} \\
\widehat{\U} \U_f N \ar[r]_-{\sigma N} & \U_f \widehat{\U} N }
\]
 is commutative. We have to show that there exists a unique morphism $\xi \colon M \rightarrow N$ with $\U_f \xi=\varphi$ and $\widehat{\U} \xi=\psi$.

 Since $M$ is $f$-regular, the morphism $\eta_f \otimes M$ is a monomorphism (see Part~(v) of Lemma~\ref{lemma:properties_of_o_f}). Thus the sequence
\begin{equation}\label{eqn:f_regular_exact_sequence}
 \xymatrix{0 \ar[r] & M \ar[r]^-{\eta_f M} & \U_f M \ar[r]^-{\pi M} & (\U_f \slash \U) M \ar[r] & 0}
\end{equation}
 is exact.

 Recall from Part~(ii) of Lemma~\ref{lemma:k_n_properties} that
\[
(\U_f \slash \U) \widehat{\eta} M \colon (\U_f \slash \U)M \rightarrow (\U_f \slash \U) \widehat{\U} M 
\]
 is an isomorphism. From the commutative diagram
\[
 \xymatrix{\U_f M \ar@{=}[r] \ar[d]_{\widehat{\eta} M} & \U_f M \ar[r]^-{\pi M} \ar[d]^{\U_f \widehat{\eta} M} & (\U_f \slash \U)M \ar[d]^{(\U_f \slash \U) \widehat{\eta} M}_{\cong} \\
 \widehat{\U} \U_f M \ar[r]_-{\sigma M} & \U_f \widehat{\U} M \ar[r]_-{\pi \widehat{\U} M} & (\U_f \slash \U) \widehat{\U} M
 }
\]
 and the analogous diagram for $N$ we deduce that the solid arrow rectangle on the right of the diagram
\[
 \xymatrix@C=40pt{0 \ar[r] & M \ar@{-->}[d]_{\xi} \ar[r]^-{\eta_f M} & \U_f M \ar[d]^{\varphi} \ar[r]^-{\pi M} & (\U_f \slash \U) M \ar[r]_-{\cong}^-{(\U_f \slash \U) \widehat{\eta} M} & (\U_f \slash \U) \widehat{\U} M \ar[d]^{(\U_f \slash \U) \psi} \\ 
0 \ar[r] & N \ar[r]_-{\eta_f N} & \U_f N \ar[r]_-{\pi N} & (\U_f \slash \U) N \ar[r]^-{\cong}_-{(\U_f \slash \U) \widehat{\eta} N} & (\U_f \slash \U) \widehat{\U} N }
\]
 is commutative. Since both rows are exact, there exists a unique $\xi \colon M \rightarrow N$ making the whole diagram commutative. Moreover, since $\eta_f M \colon M \rightarrow \U_f M$ gives a reflection into the full subcategory of $\ca{C}$ of $f$-local objects (see Part~(iv) of Lemma~\ref{lemma:properties_of_o_f}), it follows that $\U_f \xi=\varphi$.

 To conclude the proof it only remains to check that $\widehat{\U} \xi=\psi$. By Part~(i) of Lemma~\ref{lemma:tor_computations}, the image
\[
 \xymatrix{0 \ar[r] & \widehat{\U} M \ar[r]^-{\widehat{\U} \eta_f M} & \widehat{\U} \U_f M \ar[r]^-{ \widehat{\U} \pi M} & \widehat{\U} (\U_f \slash \U) M \ar[r] & 0}
\]
 of the exact sequence \eqref{eqn:f_regular_exact_sequence} under the functor $\widehat{\U} \otimes -$ is exact for any $f$-regular module $M$. In particular, $\widehat{\U} \eta_f N$ is a monomorphism. The claim that $\widehat{\U} \xi=\psi$ therefore follows from the commutative diagram
\[
 \xymatrix{\widehat{\U} M \ar[r]^-{\widehat{\U} \eta_f M} \ar[d]_{\widehat{\U} \xi} & \widehat{\U} \U_f M \ar[r]^-{\sigma M} \ar[d]^{\widehat{\U} \varphi} & \U_f \widehat{\U} M \ar[d]^{\U_f \psi} \\ 
 \widehat{\U} N \ar[r]_-{\widehat{\U} \eta_f N} & \widehat{\U} \U_f N \ar[r]_-{\sigma N} & \U_f \widehat{\U} N }
\]
 in $\ca{C}$.
 
 The fact that $\eta_f \widehat{\U} M $ is monic also shows that $\widehat{\U} M$ is $f$-regular (see Part~(v) of Lemma~\ref{lemma:properties_of_o_f}).
\end{proof}

 With this in hand, we can prove the following version of the Beauville--Laszlo theorem for lfp abelian tensor categories.

\begin{thm}\label{thm:beauville_laszlo_for_tensor_cats}
 Let $\ca{C}$ be a lfp abelian tensor category with enough flat objects. Let $f \colon I \rightarrow \U$ be an effective Cartier divisor in $\ca{C}$. Then the natural isomorphism \eqref{eqn:canonical_iso_in_base_change} induces an equivalence
\[
 \ca{C}^{f-\reg} \rightarrow \ca{C}_{\U_f} \pb{\ca{C}_{\U_f \widehat{\U}}} \ca{C}_{\widehat{\U}}^{f_{\widehat{\U}}-\reg}
\]
 of categories. This equivalence restricts to equivalences of tensor categories
\[
 \ca{C}^{\flt} \rightarrow \ca{C}^{\flt}_{\U_f} \pb{\ca{C}^{\flt}_{\U_f \widehat{\U}}} \ca{C}_{\widehat{\U}}^{\flt} 
\]
 and
\[
  \ca{C}^{\flt, \fp} \rightarrow \ca{C}^{\flt, \fp}_{\U_f} \pb{\ca{C}^{\flt, \fp}_{\U_f \widehat{\U}}} \ca{C}_{\widehat{\U}}^{\flt, \fp}
\]
 between the subcategories of flat objects, respectively finitely presentable flat objects.
\end{thm}

\begin{proof}
 First note that the restrictions are well-defined: any base change functor $(-)_B \colon \ca{C} \rightarrow \ca{C}_B$ preserves flat objects since $U(M_B \ten{B}-)\cong M\otimes U(-)$, where $U\colon \ca{C}_B \rightarrow \ca{C}$ denotes the forgetful functor. Since $U$ preserves filtered colimits, its left adjoint $(-)_B$ preserves finitely presentable objects. 

 Moreover, an object $G \in \ca{C}_{\widehat{\U}}$ is $f_{\widehat{\U}}$-regular if and only if its underlying object in $\ca{C}$ is $f$-regular. Indeed, this follows from the isomorphism $f_{\widehat{\U}} \ten{\widehat{\U}} G \cong f \otimes G$. The last statement of Lemma~\ref{lemma:beauville_laszlo_fully_faithful} therefore implies that the first functor is well-defined.

 From Lemma~\ref{lemma:beauville_laszlo_fully_faithful} we also know that all the functors in question are fully faithful, so it only remains to check that they are essentially surjective.

 Let $F \in \ca{C}_{\U_f}$ be arbitrary, let $G \in \ca{C}_{\widehat{\U}}$ be $f$-regular, and let $\alpha \colon \widehat{\U} F \rightarrow \U_f G$ be an isomorphism of $\U_f \widehat{\U}$-modules. We need to check that there exists an object $M \in \ca{C}$ together with an isomorphism $(\varphi,\psi) \colon (\U_f M,\sigma M, \widehat{\U}M) \rightarrow (F,\alpha,G)$. The idea is to construct the epimorphism in the exact sequence \eqref{eqn:f_regular_exact_sequence} and then to define $M$ to be its kernel.

 The epimorphism
\[
 \xymatrix{\widehat{\U} F \ar[r]^-{\alpha} & \U_f G \ar[r]^-{\pi G} & (\U_f \slash \U) G}
\]
 of $\widehat{\U}$-modules corresponds to a morphism $\overline{\alpha} \colon F \rightarrow (\U_f \slash \U) G$ under the adjunction $(-)_{\widehat{\U}} \dashv U \colon \ca{C} \rightarrow \ca{C}_{\widehat{\U}}$. Explicitly, the morphism $\overline{\alpha}$ is given by the composite
\[
 \xymatrix{F \ar[r]^-{\widehat{\eta}F} & \widehat{\U} F \ar[r]^-{\alpha} & \U_f G \ar[r]^-{\pi G} & (\U_f \slash \U) G }
\]
 in $\ca{C}$. It follows that the square
\[
 \xymatrix{ \widehat{\U} F \ar[d]_{\alpha}^{\cong} \ar[r]^-{\widehat{\U} \overline{\alpha}} & \widehat{\U} (\U_f \slash \U) G \ar[d]^{\widehat{\U}\text{-action}} \\
\U_f G \ar[r]_-{\pi G} & (\U_f \slash \U) G }
\]
 is commutative. Moreover, the action morphism in the above diagram is an isomorphism since $\widehat{\eta} (\U_f \slash \U) G$ is an isomorphism (see Part~(ii) of Lemma~\ref{lemma:k_n_properties}). This shows that $\widehat{\U} \overline{\alpha}$ is an epimorphism. Since $\U_f (\U_f \slash \U) \cong 0$, the morphism $\U_f \overline{\alpha}$ is tautologically an epimorphism. But $\U_f$ and $\widehat{\U}$ jointly detect epimorphisms by Lemma~\ref{lemma:faithfulness_of_o_f_and_o_hat} and Part~(i) of Lemma~\ref{lemma:almost_faithful_algebra}.

 Let $i \colon M \rightarrow F$ be the kernel of the morphism $\overline{\alpha} \colon F \rightarrow (\U_f \slash \U) G$ in $\ca{C}$. We claim that there exists an isomorphism
\[
 (\varphi,\psi) \colon (\U_f M, \sigma M,\widehat{M}) \rightarrow (F,\alpha,G)
\]
 in $\ca{C}_{\U_f} \ten{\ca{C}_{\U_f \widehat{\U}}} \ca{C}_{\widehat{\U}}$. Note that $M$ is certainly $f$-regular as subobject of the $f$-local object $F \in \ca{C}$.

 Applying $\U_f \otimes -$ to the exact sequence
\[
 \xymatrix{ 0 \ar[r] & M \ar[r]^-{i} & F \ar[r]^-{\overline{\alpha}} & (\U_f \slash \U) G \ar[r] & 0}
\]
 we find that $\U_f \otimes i$ is an isomorphism. We can compose it with the isomorphism $\U_f F \cong F$ (the action of $\U_f$ on $F$) to obtain an isomorphism $\varphi \colon \U_f M \rightarrow F$ of $\U_f$-modules.

 The bottom row in the commutative solid arrow diagram
\[
 \xymatrix{0 \ar[r] & \widehat{\U} M \ar@{-->}[d]_{\psi} \ar[r]^-{\widehat{\U} i} & \widehat{\U} F \ar[d]_{\cong}^{\alpha} \ar[r]^-{\widehat{\U} \overline{\alpha}} & \widehat{\U} (\U_f \slash \U) G \ar[d]^{\widehat{\U}\text{-action}}_{\cong} \ar[r] & 0 \\ 
 0 \ar[r] & G \ar[r]_-{\eta_f G} & \U_f G \ar[r]_-{\pi G} & (\U_f \slash \U) G \ar[r] & 0 }
\]
 is exact since $G$ is $f$-regular and the top row is exact since
\[
\Tor_1^{\ca{C}}\bigl(\widehat{\U},(\U_f \slash \U) G \bigr) \cong 0 
\]
 (see Part~(i) of Lemma~\ref{lemma:tor_computations}). Thus there exists a unique isomorphism $\psi$ of $\widehat{\U}$-modules making the above diagram commutative.

 From the definition of $\varphi$ and $\psi$ it follows that all the regions of the diagram
\[
 \xymatrix{
\widehat{\U} M \ar@{=}[rrr] \ar[dd]_{\eta_f \widehat{\U} M} \ar[rd]^{\widehat{\U} \eta_f M} & & & \widehat{\U} M \ar[ld]_{\widehat{\U} i} \ar[rd]^{\psi} \\
 & \widehat{\U} \U_f M \ar[ld]^{\sigma M} \ar[r]^-{\widehat{\U} \varphi} \ar@{}[rd]|{(\ast)} & \widehat{\U} F \ar[rd]_{\alpha} & & G \ar[ld]^{\eta_f G} \\
 \U_f \widehat{\U} M \ar[rrr]_-{\U_f \psi} & & & \U_f G
}
\]
 except possibly $(\ast)$ commute (including the outer pentagon). Thus $(\ast)$ commutes when precomposed with $\widehat{\U} \eta_f M$. But $(\ast)$ is a diagram of $\U_f$-modules whose initial vertex $\widehat{\U} \U_f M$ is a free $\U_f$-module. It follows that $(\ast)$ is commutative, hence that
\[
 (\varphi,\psi) \colon (\U_f M, \sigma M, \widehat{\U} M) \rightarrow (F,\alpha, G)
\]
 is indeed an isomorphism in $\ca{C}_{\U_f} \pb{\ca{C}_{\U_f \widehat{\U}}} \ca{C}_{\widehat{\U}}$. As we already observed, the object $M$ is $f$-regular, so this concludes the proof of the first claim.

 It remains to check that the kernel $M$ of $\overline{\alpha}$ is flat if $F$ and $G$ are flat as $\U_f$-module respectively $\widehat{\U}$-module, and that it is furthermore finitely presentable if both $F$ and $G$ are. The second claim follows from the first and Part~(iii) of Lemma~\ref{lemma:almost_faithful_algebra} (applied to the algebra $B=\U_f \times \widehat{\U}$).

 Suppose that $F$ is a flat $\U_f$-module and $G$ is a flat $\widehat{\U}$-module. Then $G$ is $f_{\widehat{\U}}$-regular by Part~(i) of Lemma~\ref{lemma:properties_of_o_hat}. Since $\U_f$ is flat, the natural isomorphism $F \otimes - \cong F \ten{\U_f } (\U_f \otimes -)$ shows that $F$ is flat as an object of $\ca{C}$. The long exact Tor-sequence arising from the exact sequence
\[
 \xymatrix{0 \ar[r] & M \ar[r]^{i} & F \ar[r]^-{\overline{\alpha}} & (\U_f \slash \U) G  \ar[r] & 0}
\]
 and an arbitrary object $C \in \ca{C}$ yields the exact sequence
\[
 \xymatrix{\Tor_2^{\ca{C}}\bigl(C, (\U_f \slash \U) G\bigr) \ar[r] & \Tor_1^{\ca{C}}(C,M) \ar[r] & \Tor_1^{\ca{C}}(C,F) } \smash{\rlap{,}}
\]
 and as we just observed we have $\Tor_1^{\ca{C}}(C,F)\cong 0$. From Part~(iii) of Lemma~\ref{lemma:tor_computations} we know that $\Tor_2^{\ca{C}}\bigl(C, (\U_f \slash \U) G\bigr) \cong 0$ as well. Thus we have $\Tor_1^{\ca{C}}(C,M) \cong 0$ for all $C \in \ca{C}$, which shows that $M$ is indeed flat.
\end{proof}

 Applying this to the category of quasi-coherent sheaves on a stack, we get the following corollary. Note that in order to get the desired basic properties (a lfp abelian tensor category), we need to make some assumptions on the stack $X$. If $X$ is \emph{algebraic} in the sense of Goerss and Hopkins \cite{GOERSS_ALGEBRAIC} ($X$ is quasi-compact with affine diagonal), then the category $\QCoh_{\fp}(X)$ is abelian and the finitely presentable quasi-coherent sheaves are closed under finite tensor products. Thus, for algebraic stacks in the sense of Goerss and Hopkins, $\QCoh(X)$ is a lfp abelian tensor category if and only if $\QCoh_{\fp}(X)$ is a generator.

\begin{cor}\label{cor:global_beauville_laszlo}
 Let $X$ be an algebraic stack and suppose that $\QCoh(X)$ is lfp and has enough flat objects. Let $Z \subseteq X$ be an effective Cartier divisor with sheaf of ideals $I \subseteq \U_X$. Let $\widehat{Z}=\Spec_X(\lim \U_X \slash I^n)$ and let $U=X \setminus Z$ be the open complement of $Z$. Then the canonical functors
\[
 \VB(X) \rightarrow \VB(U) \pb{\VB(U \pb{X} \widehat{Z} )} \VB(\widehat{Z})
\]
 and
\[
 \VB^c(X) \rightarrow \VB^c(U) \pb{\VB^c(U \pb{X} \widehat{Z} )} \VB^c(\widehat{Z}) 
\]
 are equivalences.
\end{cor}

\begin{proof}
 Since any algebraic stack has an affine $\fpqc$-atlas $X_0 \rightarrow X$ for some affine scheme $X_0$, a quasi-coherent sheaf $M$ on $X$ is a vector bundle if and only if it is finitely presentable and flat. 

 Recall that $U=X \setminus Z$ is  isomorphic to $\Spec_X(\U_f)$ where $f \colon I \rightarrow \U_X$ denotes the inclusion. Thus both morphisms to $X$ are affine, so the pullback is also affine over $X$, with quasi-coherent sheaf of algebras given by the tensor product of $\U_f$ and $\widehat{\U}=\lim \U_X \slash I^n$ (see for example \cite[Proposition~4.5]{SCHAEPPI_GEOMETRIC}). The first claim therefore follows from Theorem~\ref{thm:beauville_laszlo_for_tensor_cats}, applied to the tensor category $\ca{C}=\QCoh(X)$.

 To see the second claim, it suffices to check that a vector bundle $V \in \VB(X)$ has constant rank if $\U_f \otimes V$ has constant rank. Since line bundles are locally free of rank one we can choose an atlas $p \colon X_0=\Spec(A) \rightarrow X$ with the property that $p^{\ast} I \cong A$. This reduces the problem to the affine case: we need to check that a finitely generated projective $A$-module $V$ has constant rank if $V_f$ has constant rank for some non-zero-divisor $f \in A$.

 This was already observed in the proof of Theorem~\ref{thm:tubular}: if $V$ did not have constant rank, we could write it as direct sum $V_0 \oplus V_1$ where the $V_i$ have different ranks. Then the localization of one of the $V_i$ at $f$ would be zero, contradicting the fact that $f$ is not a zero-divisor.
\end{proof}

 If $X$ is a quasi-compact semi-separated scheme\footnote{The quasi-compact and semi-separated schemes are precisely the algebraic stacks in the sense of Goerss and Hopkins which are also schemes.}, the category of quasi-coherent sheaves has enough flat objects by \cite[\S 1.2]{TARRIO_LOPEZ_LIPMAN}; see also the detailed notes in \cite[\S 2.4]{MURFET}. Thus the above corollary is applicable if the category of quasi-coherent sheaves is also locally finitely presentable.

 The corollary is also applicable if $X$ is an Adams stack (since objects with duals are both flat and finitely presentable). Combining this with Theorem~\ref{thm:pushout_recognition} we get the following result, which shows that the infinitesimal neighbourhood of an effective Cartier divisor behaves like a tubular neighbourhood (at least in the 2-category of Adams stacks).

\begin{thm}\label{thm:global_tubular}
 Let $X$ be an Adams stack over $R$ and let $Z \subseteq X$ an effective Cartier divisor with open complement $U =X \setminus Z$ and with sheaf of ideals $I \subseteq \U_X$. Let $\widehat{Z}=\Spec_X(\lim \U_X \slash I^n)$ be the infinitesimal neighbourhood of $Z$ in $X$. Then the bicategorical pullback square\footnote{This pullback square can be computed in the category of all stacks on the $\fpqc$-site $\Aff_R$, see \cite[Corollary~4.8]{SCHAEPPI_INDABELIAN}.}
\[
 \xymatrix{U \pb{X} \widehat{Z} \ar[r] \ar[d] & \widehat{Z} \ar[d] \\ U \ar[r] & X}
\]
 is a pushout square in the 2-category $\ca{AS}$ of Adams stacks.
\end{thm}

\begin{proof}
 By Theorem~\ref{thm:pushout_recognition}, we need to check that
\[
 \VB^c(X) \rightarrow \VB^c(U) \pb{\VB^c(U \pb{X} \widehat{Z} )} \VB^c(\widehat{Z}) 
\]
 is an equivalence and that the two functors
\[
 \QCoh_{\fp}(X) \rightarrow \QCoh_{\fp}(U) \quad\text{and}\quad
 \QCoh_{\fp}(X) \rightarrow \QCoh_{\fp}(\widehat{Z})
\]
 jointly detect non-zero objects. The first claim is proved in Corollary~\ref{cor:global_beauville_laszlo} and the second is a consequence of Lemma~\ref{lemma:faithfulness_of_o_f_and_o_hat}.
\end{proof}

\bibliographystyle{amsalpha}
\bibliography{descent}

\end{document}